\documentclass[11pt]{amsart}

\usepackage{amscd,amssymb,amsmath,latexsym,enumerate}
\usepackage[utf8]{inputenc}
\usepackage{amscd,amssymb,amsmath,amsthm,bbm}
\usepackage[mathscr]{euscript}
\usepackage{mathrsfs}

\usepackage[breaklinks=true,colorlinks=true,
linkcolor=black,urlcolor=black,citecolor=black,% PRINT
bookmarks=true,bookmarksopenlevel=2]{hyperref}

\usepackage{color}

\DeclareMathAlphabet{\mathpzc}{OT1}{pzc}{m}{it}
\numberwithin{equation}{section}
\newtheorem{definition}{Definition}[section]
\newtheorem{theorem}[definition]{Theorem}

\newtheorem{proposition}[definition]{Proposition}
\newtheorem{lemma}[definition]{Lemma}

\newtheorem{corollary}[definition]{Corollary}
\newtheorem{remark}[definition]{Remark}

\theoremstyle{remark}
\newtheorem{example}[definition]{Example}
\def\Gw{\Omega}              
\newcommand{\diver}{\mathrm{div}\,}
\newcommand{\loc}{{\mathrm{loc}}}

\newcommand{\dm}{\,\mathrm{d}m}

\newcommand{\dmu}{\,\mathrm{d}\mu}

\newcommand{\core}{C_0^{\infty}(\Omega)}
\newcommand{\R}{{\mathbb R}}
\newcommand{\N}{{\mathbb N}}

\renewenvironment{proof}{{\bfseries Proof.}}{\hfill$\Box$}

\newcommand{\CM}{{\mathbb C}}
\newcommand{\NM}{{\mathbb N}}

\newcommand{\RM}{{\mathbb R}}

\newcommand{\Cc}{{\mathcal C}}
\newcommand{\Dd}{{\mathcal D}}

\newcommand{\Hh}{{\mathcal H}}

\newcommand{\ab}{{\mathbf a}}

\newcommand{\supp}{{\mathrm supp}}

\begin{document}
\title[Generalized eigenfunctions and eigenvalues]{Generalized eigenfunctions and eigenvalues: a unifying framework for Shnol-type theorems}

\author{Siegfried Beckus, Baptiste Devyver}

\address{Institut f\"ur Mathematik\\
Universit\"at Potsdam\\
Potsdam, Germany}
\email{beckus@uni-potsdam.de}

\address{Department of Mathematics\\
Technion - Israel Institute of Technology\\
Haifa, Israel}
\email{devyver@technion.ac.il}

\begin{abstract}
Let $H$ be a generalized Schr\"odinger operator on a domain of a non-compact connected Riemannian manifold, and a generalized eigenfunction $u$ for $H$: that is, $u$ satisfies the equation $Hu=\lambda u$ in the weak sense but is not necessarily in $L^2$. The problem is to find conditions on the growth of $u$, so that $\lambda$ belongs to the spectrum of $H$. We unify and generalize known results on this problem. In addition, a variety of examples is provided, illustrating the different nature of the growth conditions.\\[0.1cm]

\noindent  2000  \! {\em Mathematics  Subject  Classification.}
Primary  \! 35P05; Secondary  35B09, 35J10, 81Q10, 81Q35.\\[2mm]

\end{abstract}

%%%%%%%%%%%%%%%%%%%%%%%%%%%%%%%%%%%%%%%%%%%%%%%%%%%%%%%%%%%%%%%%%%%%

\maketitle

%%%%%%%%%%%%%%%%%%%%%%%%%%%%%%%%%%%%%%%%%%%%%%%%%%%%%%%%%%%%%%%%%%%%%%%%%%%%%%%%%%%%%
%%%%%%%%%%%%%%%%%%%%%%%%%%%%%%%%%%%%%%%%%%%%%%%%%%%%%%%%%%%%%%%%%%%%%%%%%%%%%%%%%%%%%
\section{Introduction and main result}
\label{Sec:Intro}
%%%%%%%%%%%%%%%%%%%%%%%%%%%%%%%%%%%%%%%%%%%%%%%%%%%%%%%%%%%%%%%%%%%%%%%%%%%%%%%%%%%%%
%%%%%%%%%%%%%%%%%%%%%%%%%%%%%%%%%%%%%%%%%%%%%%%%%%%%%%%%%%%%%%%%%%%%%%%%%%%%%%%%%%%%%

The paper deals with the following question: under which conditions does a generalized eigenvalue for a Schr\"odinger-type operator belong to the spectrum? Here, by ``generalized eigenvalue'' $\lambda$, we mean $\lambda\in \R$ so that there is weak solution to the equation $Hu=\lambda u$; the function $u$ is then called an ``generalized eigenfunction''. Since $u$ does not necessarily belong to $L^2$, it is not straightforward to determine whether $\lambda$ belongs to the spectrum of $H$. That this is indeed the case depends on further assumptions on the growth of the generalized eigenfunction (and typically the allowed growth depends on both $H$ and the domain). Statements giving conditions for $\lambda$ to belong to the spectrum are known by the name {\em Shnol-type theorems}, in recognition of an early work by Shnol \cite{Shn57}. There he proved that if $H=\Delta+V$ is a standard Schr\"odinger operator on $\R^d$ (whose potential $V$ satisfies certain technical conditions), and the generalized eigenfunction $u$ has at most polynomial growth, then $\lambda$ belongs to the spectrum of $H$. This celebrated result was independently rediscovered by Simon \cite{Sim81} for a more general class of potentials. In addition, Simon showed that almost every (w.r.t. the spectral measure) energy in the spectrum admits a generalized eigenfunction with at most polynomial growth. The latter result is based on an general method \cite{Bro54,CyconFroeseKirschSimon87,Shu92} for eigenfunction expansion. Remarkable generalizations of these results for subexponentially growing eigenfunctions have been proven in the setting of Dirichlet forms \cite{BoSt03,BoLeSt09,FrLeWi14}, quantum graphs \cite{Ku05} and graphs \cite{HaKe11}. %In \cite{BoLeSt09}, the polynomially growth was replaced by the subexponential growth of the generalized eigenfunction reflecting the growth of the underlying geometry.

\medskip

Recently, the following problem was raised in \cite[Conjecture~9.9]{DeFrPi14}: if instead of having subexponential growth, the generalized eigenfunction $u$ is bounded by a certain quantity, intrinsically defined by the operator $H$, can we also conclude that $\lambda$ is in the spectrum? More precisely, it was conjectured that $\lambda$ is in the spectrum, provided the generalized eigenfunction is bounded pointwise by (a multiple of) the Agmon ground state. This conjecture has been proven a few years later in \cite{BP17}. 

\medskip

Summing up, there are two sets of results on this problem:
\begin{itemize}
\item[(A)] The subexponential growth of the generalized eigenfunction implies that the associated eigenvalue belongs to the spectrum, c.f. \cite{Shn57,Sim81,Ku05,BoLeSt09,HaKe11}.
\item[(B)] The generalized eigenfunction being bounded pointwise by the (Agmon) ground state implies that the associated eigenvalue belongs to the spectrum, c.f. \cite{BP17}.
\end{itemize}

While (A) requires only $L^2$-estimates on the generalized eigenfunction, (B) requires pointwise estimates; however, one can sometimes prove that the required $L^2$ and pointwise estimates are equivalent, by means of ``mean-value-type inequalities'', see e.g. \cite{Sim81}. We also mention that both results (A) and (B) rely crucially on Caccioppoli-type estimates (see \cite{HeinonenKilpelainenMartio93,BiMo95}, and references therein for the unperturbed operator, and \cite{BoLeSt09} for the Dirichlet form setting).

\medskip

The aim of the present paper is to unify the two approaches (A) and (B) and put them in a common framework; we will also extend significantly the result (B) by allowing a more general growth on the generalized eigenfunction, still requiring pointwise estimates. As we shall see, the obtained generalization is also close to be optimal. In order to show that a generalized eigenvalue belongs to the spectrum, we will consider special Weyl sequences; these Weyl sequences will be built out of the generalized eigenfunction, and of a sequence of cut-off functions that have certain good properties, which will be called an {\em admissible cut-off sequence}. The concept of an admissible cut-off sequence will turn out to be the one unifying the results (A) and (B). As will be demonstrated in examples, our results can sometimes apply even if the generalized eigenfunction is exponentially growing. This is of interest, because there are very few Shnol-type results in the literature, that pertains to the case in which the generalized eigenfunction is not subexponentially growing. One such celebrated result is due to Brooks \cite{B81}: on a Riemannian cover of a compact manifold, consider the eigenfunction $u=1$ for the Laplacian; then, its associated eigenvalue $\lambda=0$ belongs to the spectrum, if and only if the deck transformation group of the covering is amenable. Since there exists amenable groups with exponential growth (e.g. the lamplighter group), this provides a simple example on which the results in (A) are not directly applicable. 

\medskip

{\bf Plan of the paper:} In Section~\ref{ssec-SettingResults}, a short introduction of the setting is provided and our main results are stated. In Section~\ref{Sec:Prel}, we review some key concepts from criticality theory, that will be needed later on. Section~\ref{Sec:Cacc} is devoted to some new Caccioppoli-type estimates, which are the key ingredients for the proof of the main result. In Section~\ref{Sec:PfMain}, we prove our main result. In Section~\ref{Sec:Crit}, we explain how to use our main result in order to recover (B); this requires building an admissible cut-off sequence whenever the underlying operator is critical, and it is achieved by using the so-called {\em Evans potential} for the operator. In Section~\ref{Sec:subexp}, we show briefly how (A) follows from our main result as well. Finally, we discuss some examples of applications of our main result in Section~\ref{Sec:Examples}.

\subsection{Setting and main results}
\label{ssec-SettingResults}
%%%%%%%%%%%%%%%%%%%%%%%%%%%%%%%%%%%%

For the purpose of this work, $\Omega$ is a domain in $\RM^d$ (or a domain in a non-compact $d$-dimensional connected Riemannian manifold). Fix a strictly positive measurable function $m$ on $\Omega$ satisfying that $m$ and $m^{-1}$ are bounded on any compact subset of $\Omega$. Define $dm:=m(x)dx$ where $dx$ is the Riemannian volume form on $\Omega$.

\medskip

We denote by $\mathit{End}(T\Omega)$ the bundle of endomorphisms of the tangent bundle $T\Omega$. The inner product and its induced norm on $T\Omega$ is denoted by $\langle\cdot,\cdot\rangle$ and $|\cdot|$. Throughout this work $A$ is a symmetric measurable section on $\Omega$ of $\mathit{End}(T\Omega)$ that is locally uniformly elliptic, that is for each $K\subseteq \Omega$ compact that there is a constant $\lambda_K\geq 1$ such that
\begin{equation}\label{Eq:A}
\frac{1}{\lambda_K} |\xi|^2 
	\leq \langle A(x)\xi,\xi\rangle
	\leq \lambda_K |\xi|^2
	\,,\qquad x\in K \text{ and } (x,\xi)\in T\Omega\,.
\end{equation}
%%%%%%%%%%%%%%%%%%
Let $L^p(\Omega,dm)$ be the associated $L^p$-space with $\Omega$. Furthermore, $L^p_{loc}(\Omega,dm)$ denotes the set of measurable $f:\Omega\to\CM$ such that $f|_K\in L^p(\Omega,dm)$ for each $K\subseteq \Omega$ compact. The set of compact, smooth functions on $\Omega$ is denoted by $\core\subseteq L^2(\Omega,dm)$. Throughout this work $\|\cdot\|_{2,m}$ denotes the $L^2$-norm on $L^2(\Omega,dm)$ and $\langle\cdot,\cdot\rangle_{2,m}$ is the corresponding inner product.

\medskip

Denote by $\nabla$ the gradient with respect to the Riemannian metric. Let $p>\frac{d}{2}$ and $V\in L^p_\loc(\Omega,dm)$ be real-valued, and $A$ be a symmetric measurable section on $\Omega$ of $\mathit{End}(T\Omega)$ satisfying \eqref{Eq:A}. Then the symmetric sesquilinear form $Q:\core\times\core\to\CM$ is defined by
%%%%%%%%%%%%%%%%%%
\begin{equation}\label{Eq:ForSchr}
Q(v,w)
	:= \int_\Omega \langle A\nabla u\,,\nabla v\rangle + V u\overline{v}  \dm\,.
\end{equation}
%%%%%%%%%%%%%%%%%%
Throughout this work, it is assumed that $Q$ is semibounded, namely $Q(v,v)\geq -c\|v\|_{2,m}^2$ for all $v\in\core$ and some $c\geq 0$. The quadratic form $Q$ is then closable, and we will consider its closure (also denoted $Q$ for simplicity), see also Remark~\ref{Rem:ClosFor}. Its domain is $\Dd(Q):=\overline{\core}^{\|\cdot\|_Q}$ where the $Q$-norm is defined by 
%%%%%%%%%%%%%%%%%%
$$\|v\|_Q
	:= \sqrt{Q(v,v)+(1+c)\|v\|_{2,m}^2}\,.
$$
%%%%%%%%%%%%%%%%%%
Thus, $\core$ is a core of $\Dd(Q)$.
There exists a unique associated self-adjoint operator $H$ associated with $Q$, which has the formal form
%%%%%%%%%%%%%%%%%%
\begin{equation}\label{Eq:Schro}
H=-\mathrm{div}(A\nabla\cdot)+V\,.
\end{equation}
%%%%%%%%%%%%%%%%%%
Here $-\mathrm{div}$ denotes the formal adjoint of the gradient with respect to the measure $m$. In order to shorten the notation, we use $\langle A\nabla v,\nabla v\rangle=:|\nabla v|_A^2$.

\medskip

\medskip

If $u\in W^{1,2}_{loc}(\Omega)$ and $v\in\core$, then $Q(u,v)$ is well-defined. With this at hand, $u\in W^{1,2}_{loc}(\Omega)$ is called {\em generalized eigenfunction of $H$ with eigenvalue $\lambda\in\RM$} if 
%%%%%%%%%%%%%%%%%%
$$
Q(u,v)=\lambda \int_\Omega u\overline{v}  \dm
\,,\qquad \text{ for all } v\in\core\,.
$$
%%%%%%%%%%%%%%%%%%
It follows from elliptic regularity (see \cite[Theorem 8.22]{GT}) that $u$ is locally H\"older continuous. The aim of this work is to find those growth conditions on a generalized eigenfunction such that its associated eigenvalue belongs to the spectrum of $H$.

%%%%%%%%%%%%%%%%%%
\begin{definition}\label{Def:AdmSeq}
Let $\varphi \in C(\Omega)\cap W_{loc}^{1,2}(\Omega)$ be a positive function. A sequence $\{\varphi_n\}_{n\in\mathbb{N}}$ of  functions in $C_0(\Omega)\cap \mathcal{D}(Q)$, is called an {\em admissible cut-off sequence for $(H,\varphi)$} if the following conditions hold:

\begin{itemize}

\item[(i)] For every $n\in\mathbb{N}$, $0\leq \varphi_n\leq \varphi$. 

%\item[(ii)] $\Omega=\bigcup_{n\in\mathbb{N}} \supp(\varphi_n)$. 

\item[(ii)] For every $n\in\mathbb{N}$, $\varphi_{n+1}(x)=\varphi(x)$, for all $x\in \supp(\varphi_n)$. 

\item[(iii)] There is a constant $C>0$ satisfying the following {\em (weak Hardy inequality)}

\begin{equation}\label{Eq:WH}\tag{wH}
\int_\Omega |v|^2\left |\nabla \left(\frac{\varphi_n}{\varphi}\right)\right|_A^2 \dm  \leq C\|v\|_Q^2
	\,,\qquad v\in \Dd(Q)\,.
\end{equation}
\end{itemize}
\end{definition}

%%%%%%%%%%%%%%%%%%

%%%%%%%%%%%%%%%%%%
\begin{example}

The two most prominent examples of an admissible cut-off sequence are built up using either distance functions for the so-called intrinsic metric (see Section \ref{Sec:subexp}), or a special null-sequence (see Section \ref{Sec:Crit}).

\end{example}
%%%%%%%%%%%%%%%%%%
Since the spectrum of $H$ is only bounded from below but not necessarily non-negative, it will be sometimes useful to shift it. To this purpose, we will sometimes consider the operator $H+W$, where $W:\Omega\to \R$ is a bounded potential, such that $H+W$ is non-negative. By the Allegretto-Piepenbrink theorem (see e.g. \cite{Agm83}), there then exists a positive function $h$ such that

$$(H+W)h=0$$
in the weak sense (equivalently, $h$ is a generalized positive eigenfunction of $H+W$ with eigenvalue zero). By elliptic regularity, the function $h$ is locally H\"older continuous.

%%%%%%%%%%%%%%%%%%
\begin{theorem}\label{thm:main}
Let $(\Omega,\dm)$ be a weighted manifold and $H$ be a Schr\"odinger-type operator on $(\Omega,\dm)$ of the form \eqref{Eq:Schro}. Let $\varphi \in C(\Omega)\cap W_{loc}^{1,2}(\Omega)$, and $\{\varphi_n\}_{n\in\mathbb{N}}$ be an admissible cut-off sequence for $(H,\varphi)$. Let $u\in W^{1,2}_{\loc}(\Omega)$ be a generalized eigenfunction of the operator $H$, associated with the eigenvalue $\lambda\in \R$. Let $A_n$ be the support of $\varphi_{n+1}(\varphi-\varphi_{n-1})$. Suppose that one of the following growth conditions on $u$ holds:

\begin{itemize}

\item[(i)] there is a bounded potential $W:\Omega\to \R$ such that $(H+W)\varphi=0$ in the weak sense, and

$$\liminf_{n\to\infty} 
	\frac{\max_{A_n}\left|\frac{u}{\varphi}\right|}{\|\varphi_n\,\frac{u}{\varphi}\|_{2,m}}
		\;\left( \int_\Omega \left|\nabla \left(\frac{\varphi_n}{\varphi}\right)\right|^2_A\, \varphi^2\dm \right)^{1/2}
			= 0\,,
$$

\item[(ii)] $\varphi$ is constant, and $$\liminf_{n\to\infty}
	\frac{\|u\|_{L^2(A_n,\dm)}
		+ \sqrt{\int_\Omega |u|^2(|\nabla \varphi_{n-1}|^2_A
			+|\nabla \varphi_{n}|_A^2
			+|\nabla \varphi_{n+1}|_A^2)\,\dm}}{\|\varphi_n u\|_{2,m}}
		=0.
$$

\end{itemize}
Then, $\lambda\in\sigma(H)$.
\end{theorem}
%%%%%%%%%%%%%%%%%%

%%%%%%%%%%%%%%%%%%
\begin{remark}
{\em
(a) Theorem~\ref{thm:main} unifies and generalizes the results of (A) and (B). To be more precise, (A) is recovered from (ii) by requiring that $|\nabla \varphi_n|$ is bounded, see Section \ref{Sec:subexp}. Furthermore, (B) is a special case of condition (i). The requirements on $\{\varphi_n\}_{n\in \N}$ are not restrictive for critical operators. More precisely, a special null sequence is constructed using the so-called Evans potential, see Section \ref{Sec:Crit}. 

\vspace{.1cm}

(b) The conditions (i) and (ii) in Theorem~\ref{thm:main} are different, as shown in Examples \ref{Ex:parabolic} and \ref{Ex:i-ii}.

\vspace{.1cm}

(c) Even though (A) can be recovered by Theorem~\ref{thm:main}~(ii), the present formulation is more general. 

\vspace{.1cm}

(d) In all fairness, our proof in the case where assumption $(ii)$ is satisfied in Theorem~\ref{thm:main}, follows by arguments which are very similar to those developped in \cite{BoLeSt09}.
}
\end{remark}
%%%%%%%%%%%%%%%%%%

%%%%%%%%%%%%%%%%%%
\begin{remark}
\label{Rem:ResOpt}
{\em
The proof of Theorem~\ref{thm:main} consists of two ingredients: first, using a well-known trick called the {\em ground state transform} (or sometimes, $h$-transform or Doob transform), one reduces the proof to the case $H$ is non-negative, and the positive function $\varphi$ is equal to $1$. Then, under these assumptions, the proof consists in showing that (a subsequence of) the sequence of compactly supported functions $\{\psi_n\}_{n\in\mathbb{N}}$, $\psi_n=\frac{\varphi_nu}{\|\varphi_nu\|_{2,m}}$ is a generalized Weyl sequence for $H$, in the sense that
%%%%%%%%%%%%%%%%%%
$$
\lim_{n\to\infty}\sup_{\|v\|_q\leq 1}\Big| Q(\psi_n,v)-\lambda \langle\psi_n,v\rangle_{2,m}\Big|=0\,,
$$
%%%%%%%%%%%%%%%%%%
c.f. Proposition~\ref{Prop:CritWeyl}. More precisely, it follows from the generalized Weyl criterion that the existence of a sequence $\{\psi_n\}_{n\in\mathbb{N}}$ in $\Dd(Q)$ satisfying the above criterion is equivalent to $\lambda$ being in the spectrum of $H$. 
}
\end{remark}

It is important to notice that in the case $\varphi=1$, the conditions (i) and (ii) in Theorem \ref{thm:main} are close to being necessary, for (a subsequence of) the sequence $\{\psi_n\}_{n\in\mathbb{N}}$, $\psi_n=\frac{\varphi_nu}{\|\varphi_nu\|_{2,m}}$ to be a generalized Weyl sequence for $H$. Indeed, since $\|\psi_n\|_{2,m}=1$, a duality argument implies that if 

$$
\lim_{n\to\infty}\sup_{\|v\|_q\leq 1}\Big| Q(\psi_n,v)-\lambda \langle\psi_n,v\rangle_{2,m}\Big|=0\,,
$$
then actually $\|\psi_n\|_Q\leq C$ for some constant $C>0$ independent of $n$. Hence, by taking $v=C^{-1}\psi_n$, we see that
%%%%%%%%%%%%%%%%%%
$$
\lim_{n\to\infty}\frac{1}{\|\varphi_nu\|_{2,m}^2}\Big| Q(\varphi_nu,\varphi_nu)-\lambda\langle \varphi_nu,\varphi_nu\rangle_{2,m} \Big|=0,
$$
%%%%%%%%%%%%%%%%%%
is a necessary condition for $\{\psi_n\}_{n\in\mathbb{N}}$ to be a generalized Weyl sequence for $H$. A standard integration by parts argument, using that $Hu=\lambda u$, shows that
%%%%%%%%%%%%%%%%%%
$$Q(\varphi_nu,\varphi_nu)-\lambda\langle \varphi_nu,\varphi_nu\rangle_{2,m}
	= \int_\Omega |u|^2|\nabla \varphi_n|_A^2 \dm\,.
$$
%%%%%%%%%%%%%%%%%%
Hence, if $\{\psi_n\}_{n\in\mathbb{N}}$, $\psi_n=\frac{\varphi_nu}{\|\varphi_nu\|_{2,m}}$ is a generalized Weyl sequence for $H$, then necessarily
%%%%%%%%%%%%%%%%%%
\begin{equation}\label{Eq:gen_Weyl}
\lim_{n\to\infty}\frac{\int_\Omega |u|^2|\nabla \varphi_n|_A^2\dm}{\|\varphi_nu\|_{2,m}^2}=0.
\end{equation}
Note that if $u\equiv1$, \eqref{Eq:gen_Weyl} is just the well-known characterization of the (first) eigenvalue $\lambda=0$ in terms of Rayleigh quotient.
%%%%%%%%%%%%%%%%%%
The conditions (i) with $\varphi\equiv 1$, as well as (ii) in Theorem~\ref{thm:main} obviously imply \eqref{Eq:gen_Weyl}. In particular we see that condition (i) in Theorem \ref{thm:main} is close to being necessary for $\lambda$ to belong to the spectrum of $H$. More precisely, one has the following result:

\begin{corollary}\label{Harnack}
Let $(\Omega,\dm)$ be a weighted manifold. Let $H$ be a Schr\"odinger-type operator on $(\Omega,\dm)$ of the form \eqref{Eq:Schro} with $V\in L^\infty(\Omega)$. Let $\{\varphi_n\}_{n\in\mathbb{N}}$ be an admissible cut-off sequence for $(H,1)$. Let $u\in W^{1,2}_{\loc}(\Omega)$ be a generalized eigenfunction of the operator $H$, associated with the eigenvalue $\lambda\in \R$. Let $A_n$ be the support of $\varphi_{n+1}(\varphi-\varphi_{n-1})$. Assume that the function $|u|$ satisfies a uniform Harnack inequality on the sets $A_n$: there is a constant $C>0$, such that for every $n\in\mathbb{N}$,

$$\sup_{A_n}|u|\leq C\inf_{A_n}|u|.$$
Then, the sequence $\{\psi_n\}_{n\in\mathbb{N}}$, $\psi_n=\frac{\varphi_nu}{\|\varphi_nu\|_{2,m}}$ is a generalized Weyl sequence for $H$ associated with the eigenvalue $\lambda$, {\em if and only if}

$$\liminf_{n\to\infty}\frac{\int_\Omega |u|^2|\nabla \varphi_n|_A^2\dm}{\|\varphi_nu\|_{2,m}^2}=0.$$
\end{corollary}

%%%%%%%%%%%%%%%%%%%%%%%%%%%%%%%%%%%%%%%%%%%%%%%%%%%%%%%%%%%%%%%%%%%%%%%%%%%%%%%%%%%%%
%%%%%%%%%%%%%%%%%%%%%%%%%%%%%%%%%%%%%%%%%%%%%%%%%%%%%%%%%%%%%%%%%%%%%%%%%%%%%%%%%%%%%
\section{Preliminaries}
\label{Sec:Prel}
%%%%%%%%%%%%%%%%%%%%%%%%%%%%%%%%%%%%%%%%%%%%%%%%%%%%%%%%%%%%%%%%%%%%%%%%%%%%%%%%%%%%%
%%%%%%%%%%%%%%%%%%%%%%%%%%%%%%%%%%%%%%%%%%%%%%%%%%%%%%%%%%%%%%%%%%%%%%%%%%%%%%%%%%%%%

%%%%%%%%%%%%%%%%%%%%%%%%%%%%%%%%%%%%%%%%%%%%%%%%%%%%%%%%%%%%%%%%%%%%%%%%%%%%%%%%%%%%%
\subsection{Forms and Weyl sequences}
\label{Ssec:FormWeyl}
%%%%%%%%%%%%%%%%%%%%%%%%%%%%%%%%%%%%%%%%%%%%%%%%%%%%%%%%%%%%%%%%%%%%%%%%%%%%%%%%%%%%%

The main idea for proving a Shnol-type theorem is to construct a Weyl-sequence for the corresponding operator. The considered operators are defined via a form. In light of this, it is convenient to work with a Weyl-criteria for forms and not for the operators, which is presented now.

\medskip

A map $Q:\Dd(Q)\times \Dd(Q)\to\CM$ is called a {\em (sesquilinear) form} defined on a linear subspace $\Dd(Q)$ of a (complex) Hilbert space $\Hh$ if $Q$ is linear in the first component and complex linear in the second. The inner product on $\Hh$ is denoted by $(\cdot,\cdot)$ and $\|\cdot\|$ denotes its induced norm. If $Q(v,w)=\overline{Q(w,v)}$ holds for all $v,w\in\Dd(Q)$, then $Q$ is called symmetric. Here $\overline{z}$ denotes the complex conjugate of the complex number $z\in\CM$. Throughout this work it is assumed that the symmetric form $Q$ is {\em semibounded}, i.e. there is a constant $c\in\RM$ such that $Q(v,v)\geq c \|v\|^2$ for all $v\in\Dd(Q)$. Following \cite{Stollmann01}, $\|\cdot\|_Q:\Dd(Q)\to[0,\infty)$ defined by 
%%%%%%%%%%%%%%%%%%
$$\|v\|_Q
	:= \sqrt{Q(v,v) + (1-c)\|v\|^2} 
	\,,\qquad v\in\Dd(Q).
$$ 
%%%%%%%%%%%%%%%%%%
It is a norm on $\Dd(Q)$, satisfying the parallelogram law, hence has an associated inner product. The form $Q$ is called {\em closed} if $(\Dd(Q),\|\cdot\|_Q)$ is a Hilbert space. 

%%%%%%%%%%%%%%%%%%
\begin{remark}\label{Rem:ClosFor}
{\em Starting from a symmetric, semibounded form $Q_0$ defined on $\Dd(Q_0)\subseteq \Hh$, a closed form is defined as follows \cite{Stollmann01}: Define
%%%%%%%%%%%%%%%%%%
\begin{align*}
\Dd(Q)
	:= &\big\{ u\in \Hh \,:\, 
			\exists \{u_n\}\subset\Dd(Q_0) \text{ s.t. } u_n\to u \text{ in } \Hh 
			\text{ and } Q_0(u_n-u_m)\to 0 \text{ if } n,m\to\infty
		\big\}\,,\\
Q(u) := &\lim_{n\to\infty} Q_0(u_n)\,,
\end{align*}
%%%%%%%%%%%%%%%%%%
which is a well-defined, closed form. Throughout this work various symmetric, semibounded, closed forms are defined in this way.
}
\end{remark}
%%%%%%%%%%%%%%%%%%

Each closed, symmetric, semibounded sesquilinear form $Q$ admits a unique self-adjoint operator $H$ with operator domain $\Dd(H)\subseteq \Dd(Q)$ satisfying $(Hv,w)=Q(v,w)$ for $v\in\Dd(H)$ and $w\in\Dd(Q)$. The corresponding {\em spectrum} of the operator $H$ is denoted by $\sigma(H)$. Each closed, symmetric, semibounded, sesquilinear form defines a quadratic form by $Q(v):=Q(v,v)$  for $v\in\Dd(Q)$.

\medskip

A proof of the following Weyl-sequence criterion  can be found in \cite{DeDuVi98,Stollmann01,BoLeSt09,KrLu14,CL14,BP17}.

%%%%%%%%%%%%%%%%%%
\begin{proposition}[\cite{DeDuVi98}]
\label{Prop:CritWeyl}
Let $Q:\Dd(Q)\times \Dd(Q)\to\CM$ be a closed, symmetric, semibounded  sesquilinear form with associated self-adjoint operator $H$. Then the following assertions are equivalent:
\begin{itemize}
\item[(i)] $\lambda\in\sigma(H)$
\item[(ii)] There exists a sequence $\{w_n\}_{n\in\NM}\subseteq\Dd(Q)$ with $\lim\limits_{n\to\infty}\|w_n\|=1$ satisfying
%%%%%%%%%%%%%%%%%%
\begin{equation}\label{Eq:Weyl}
\liminf_{n\to\infty} \;
	\sup_{v\in\Dd(Q), \|v\|_Q\leq 1}
		\Big|
			Q(w_n,v)-\lambda  ( w_n,v)
		\Big|
			= 0.
\end{equation}
%%%%%%%%%%%%%%%%%%
\end{itemize}
\end{proposition}
%%%%%%%%%%%%%%%%%%

\begin{remark}
{\em 
Actually, the original statement of \eqref{Eq:Weyl} in Proposition \ref{Prop:CritWeyl} is with a limit instead of a liminf; however, passing to a subsequence, the statement with the liminf is easily obtained.
}
\end{remark}

%%%%%%%%%%%%%%%%%%%%%%%%%%%%%%%%%%%%%%%%%%%%%%%%%%%%%%%%%%%%%%%%%%%%%%%%%%%%%%%%%%%%%
\subsection{Criticality theory}
\label{Ssec:PrelCrit}
%%%%%%%%%%%%%%%%%%%%%%%%%%%%%%%%%%%%%%%%%%%%%%%%%%%%%%%%%%%%%%%%%%%%%%%%%%%%%%%%%%%%%

In the following, a reminder of the criticality theory as well as the ground state transform of an operator $H=-\mathrm{div}(A\nabla\cdot)+V$ is provided. Throughout this section $Q$ denotes the form given in \eqref{Eq:ForSchr} and $H$ is its unique self-adjoint operator.

\medskip

We say $H$ is {\em supercritical in $\Omega$}, if $H$ is not nonnegative. Furthermore, $H$ is called {\em critical in $\Omega$} if $H\geq 0$ and for each nonnegative $W\in L^p_{loc}(\Omega,\dm)$, with $p>\frac{d}{2}$, the operator $H-W$ is supercritical. Otherwise, $H$ is called {\em subcritical}. As explained below, each critical operator admits a unique (up to a multiplicative constant) $H$-harmonic function, which is called (Agmon) ground state.

\medskip

Consider the Sobolev space $W^{1,2}(\Omega)$ of functions in $L^2(\Omega,\dm)$ admitting weak derivatives up to order $1$ in $L^2(\Omega,\dm)$. Let $W^{1,2}_{loc}(\Omega)$ be the set of measurable $f:\Omega\to\CM$ such that $f|_K\in W^{1,2}(\Omega)$ for each compact $K\subseteq\Omega$. Note that $Q(u,v)$ is well-defined for $u\in W^{1,2}(\Omega)$ and $v\in\core$ as $v$ has compact support. An element $u\in W^{1,2}_{loc}(\Omega)$ is  called {\em $H$-(super)harmonic in $\Omega$} if $Q(u,v)=0$ ($Q(u,v)\geq0$) holds for all $v\in\core$. Denote by $\Cc_H(\Omega)$ the cone of all positive $H$-harmonic functions in $\Omega$. 

\medskip

We write $K\Subset\Omega$ whenever $\overline{K}$ is compact and $K\subseteq\Omega$. Let $K\Subset\Omega$ and $h$ be a positive $H$-harmonic function in $\Omega\setminus K$. Then $h$ is called {\em positive $H$-harmonic of minimal growth at infinity in $\Omega$} if for all $K\Subset K'\Subset\Omega$ with smooth boundary and each $H$-superharmonic $v\in\Cc(\overline{\Omega\setminus K'})$ satisfying $h\leq v$ on the boundary $\partial K'$, the estimate $h\leq v$ holds in $\Omega\setminus K'$. Then $h\in\Cc_H(\Omega)$ is called the {\em (Agmon) ground state} if $h$ has minimal growth at infinity (it can be shown that it is unique up to a multiplicative constant). 

\medskip

Suppose $H\geq 0$. A sequence of non-negative functions $\{\varphi_n\}_{n\in\NM}\subseteq\core$ is called a {\em null-sequence} if there is a ball $B\Subset\Omega$ satisfying, for some constant $c>0$, 
%%%%%%%%%%%%%%%%%%
$$\int_B\varphi^2_n \dm=c\,,\; n\in\NM\,,
	\qquad\text{ and }\qquad
	\lim_{n\to\infty} Q(\varphi_n,\varphi_n)=0\,.
$$ 
%%%%%%%%%%%%%%%%%%
With this at hand, $h\in W^{1,2}_{loc}(\Omega)$ is called a {\em null-state of $Q$} if $h$ is strictly positive and there is a null-sequence $\{\varphi_n\}_{n\in\NM}$ that converges in $L^2_{loc}(\Omega)$ to $h$.

\medskip

There are various characterizations of criticality which are provided in the following statement. The proof of these results can be found in \cite{Pi07,PiTi06,Pinsky95,KePiPo16} and references therein.

%%%%%%%%%%%%%%%%%%
\begin{theorem}[Criticality characterization]\label{Theo:char_criticality}
Let $Q$ be the form given in \eqref{Eq:ForSchr} and $H$ its unique self-adjoint operator of the form \eqref{Eq:Schro}. If $Q$ is nonnegative on $\core$, then the following assertions are equivalent:
%%%%%%%%%%%%%%%%%%
\begin{itemize}
  \item [(i)] $H$ is critical in $\Omega$.
  \item [(ii)]  $H$ admits an (Agmon) ground state in $\Omega$.
  \item [(iii)] $H$ admits a unique  (up to a multiplicative constant) positive $H$-super\-harmonic function  in $\Omega$.
\item [(iv)]  For every open ball $B\Subset \Gw$, there is a null-sequence $\{h_{n}\}_{n\in\NM}$ such that $\int_B h_{n}(x)^2\dm=1$ for all $n\geq 0$.
\item [(v)] There  exists a null-sequence $\{h_{n}\}_{n\in\NM}$ satisfying $0\leq h_{n}\leq h$ in $\Gw$, where $h$ is a positive $H$-harmonic function on $\Gw$, \!and $h_{n}(x) \to h(x)$  locally uniformly in $\Gw$.
\end{itemize}
%%%%%%%%%%%%%%%%%%
In particular, $h$ is a null-state if and only if it is an (Agmon) ground state.
\end{theorem}
%%%%%%%%%%%%%%%%%%

Let $h\in W^{1,2}_{loc}(\Omega)\cap \Cc(\Omega)$ be a strictly positive function. Define the operator
%%%%%%%%%%%%%%%%%%
$$T_h:L^2(\Omega,\dm)\to L^2(\Omega,\dmu)\,,\quad
T_h(v):= \frac{v}{h}\,,
$$
%%%%%%%%%%%%%%%%%%
where $\dmu:= h^2\dm$.
Clearly, $T_h$ is invertible and $T_h^{-1}=T_{h^{-1}}$. Suppose that $h$ is a positive function such that $(H+W)h=0$ where $W\in\Cc(\Omega)$ is non-negative and bounded. Define a self-adjoint operator

$$H_h 
	:= T_h\circ H \circ T_h^{-1},\quad \mathcal{D}(H_h)=T_h\mathcal{D}(H)
$$
and its associated quadratic form quadratic form

$$Q_h(u,v) = Q(T_h^{-1}u,T_h^{-1}v) ,\quad u,v\in T_h\mathcal{D}(Q).$$
Since $C_0^\infty(\Omega)$ is a core for $Q$, it follows that $T_hC_0^\infty(\Omega)$ is a core for $Q_h$.
%%%%%%%%%%%%%%%%%%
Note that $Q_h$ and $Q$ are both semibounded with the same constant. Without loss of generality, let
$$
\|v\|_Q = \sqrt{Q(v,v)+(1+\|W\|_\infty)\|v\|_{2,m}^2}\,,\qquad \|v\|_{Q_h} = \sqrt{Q_h(v,v)+(1+\|W\|_\infty)\|v\|_{2,\mu}^2}.
$$
Since $T_h$ is an isometry on the $L^2$-spaces, $\|v\|_Q=\|T_hv\|_{Q_h}$ holds by definition for all $v\in\core$. Thus, 
$$
T_h:\big(\Dd(Q),\|\cdot\|_Q\big)\to \big(\Dd(Q_h),\|\cdot\|_{Q_h}\big)
$$ 
is a surjective isometry. Note that since $A$ is locally uniformly elliptic, a sequence $\{\varphi_n\}_{n\in\N}\subset\core$ converges in $W^{1,2}_{loc}(\Omega)$ to $\varphi\in W^{1,2}_{loc}(\Omega)\cap \Cc(\Omega)$, if for each $v\in L^\infty$ with compact support,
$$
\lim_{n\to\infty}\int_\Omega |v|^2\left( |\nabla(\varphi_n-\varphi)|_A^2+|\varphi_n-\varphi|^2\right)\dm=0
$$

%%%%%%%%%%%%%%%%%%
\begin{lemma}
\label{Lem:W12conv}
If $\{\varphi_n\}_{n\in\N}\subset\core$ converges to $\varphi$ in $W^{1,2}_{loc}(\Omega)$, then for every $v\in\core$, we have $\lim_{n\to\infty}\|v(\varphi_n-\varphi)\|_{Q_h}$.
\end{lemma}
%%%%%%%%%%%%%%%%%%

%%%%%%%%%%%%%%%%%%
\begin{proof}
Let $v\in\core$. First note that
%%%%%%%%%%%%%%%%%%% 
$$\|v(\varphi_n-\varphi)\|_{2,\mu}^2
	\leq ||h||^2_{L^\infty(supp(v))} \int_\Omega |v|^2|\varphi_n-\varphi|^2 \dm
$$
%%%%%%%%%%%%%%%%%%%
which converges to zero. Furthermore, a short computation yields
%%%%%%%%%%%%%%%%%%%
\begin{align*}
|\nabla\big(v(\varphi_n-\varphi)\big)|_A^2
	= &|v|^2 |\nabla(\varphi_n-\varphi)|_A^2 + v (\varphi_n-\varphi) \langle A\nabla (\varphi_n-\varphi),\nabla v\rangle\\
	&+ \overline{v} (\varphi_n-\varphi) \langle A\nabla v,\nabla (\varphi_n-\varphi)\rangle
		+ |\varphi_n-\varphi|^2 |\nabla v|_A^2\,.
\end{align*}
%%%%%%%%%%%%%%%%%%%
Hence, 
%%%%%%%%%%%%%%%%%%%
\begin{align*}
&|Q_h(v(\varphi_n-\varphi))|\\
	\leq &\int_\Omega|\nabla\big(v(\varphi_n-\varphi)\big)|_A^2 h^2\dm + \int_\Omega W|v|^2 |(\varphi_n-\varphi)|^2 h^2\dm\\
	\leq &C\left(2	\int_\Omega |v|^2|\nabla(\varphi_n-\varphi)|_A^2\dm 
			+ 2 \int_\Omega |\nabla v|_A^2 |\varphi_n-\varphi|^2\dm
			+ ||W||_\infty\int_\Omega |v|^2 |(\varphi_n-\varphi)|^2\dm\right)^2,
\end{align*}
with $C=||h||^2_{L^\infty(supp(v))}$, and where $2ab\leq a^2+b^2$ was used in the last estimate. One then concludes that $|Q_h(v(\varphi_n-\varphi))|$ converges to zero.
\end{proof}
%%%%%%%%%%%%%%%%%%

%%%%%%%%%%%%%%%%%%
\begin{proposition}[Ground state transform]\label{Prop:GrouStaTransf}
Let $Q$ be a semibounded form as defined in \eqref{Eq:ForSchr} with its associated self-adjoint operator $H$. Consider a bounded $W:\Omega\to\R$ such that $H+W\geq 0$ and let $h\in\Cc(\Omega)\cap W^{1,2}_{loc}(\Omega)$ be a positive $(H+W)$-harmonic function. Then the following assertions hold.
\begin{itemize}

\item[(i)] The following formula holds: for every $u,v$ in $W^{1,2}(\Omega)$ with compact support,
%%%%%%%%%%%%%%%%%%
$$Q_h(u,v) 	= \int_\Omega \langle A\nabla u,\nabla v\rangle \, d\mu - \int_\Omega W u\overline{v}\dmu\,.
$$
%%%%%%%%%%%%%%%%%%

\item[(ii)] $C_0^\infty(\Omega)\subset \mathcal{D}(Q_h)$ is a core for $Q_h$.

%The vector space $T_h\core$ is a core in $\Dd(Q_h)$ and $T_h^{-1}\core$ is a core in $\Dd(Q)$. Furthermore, $T_h:\big(\Dd(Q),\|\cdot\|_Q\big)\to\big(\Dd(Q_h),\|\cdot\|_{Q_h}\big)\,,\, v\mapsto \frac{v}{h}\,,$ is an isometry.

\item[(iii)] The spectra $\sigma(H)$ and $\sigma(H_h)$ coincide as subsets of $\RM$.

\item[(iv)] If $H$ is critical, $W\equiv 0$ and $h$ is the Agmon ground state of $H$, then $H_h$ is critical with (Agmon) ground state $1$.
\end{itemize}

\end{proposition}
%%%%%%%%%%%%%%%%%%

%%%%%%%%%%%%%%%%%%
\begin{proof}
(i) A short computation implies the result, using that $h$ is $H+W$-harmonic.

\medskip

(ii) First we show that $C_0^\infty(\Omega)\subset \mathcal{D}(Q_h)$. Let $v\in\core$ and $\{h_n\}_{n\in\N}\subset\core$ be such that it converges to $h$ in $W^{1,2}_{loc}(\Omega)$. Define $v_n:=\frac{h_n v}{h}\in T_h\core$. One has
$$
\nabla\frac{h_n}{h} = \frac{(h\nabla h_n - h \nabla h) + (h\nabla h - h_n\nabla h)}{h^2}.
$$
Hence, $\{\frac{h_n}{h}\}_{n\in\N}$ converges to $1$ in $W^{1,2}_{loc}(\Omega)$. Applying Lemma~\ref{Lem:W12conv}, we derive $\lim_{n\to\infty}\|v_n-v\|_{Q_h}=0$. As $T_h\core\subset \Dd(Q_h)$, $v\in \Dd(Q_h)$ follows.

\medskip

Since $\frac{1}{h}\in W^{1,2}_{loc}(\Omega)\cap\Cc(\Omega)$, let $\{h'_n\}_{n\in\N}\subset\core$ be such that it converges to $\frac{1}{h}$ in $W^{1,2}_{loc}(\Omega)$. Let $u\in\core$. Then $\lim_{n\to\infty}\|h'_nu-\frac{u}{h}\|_{Q_h}=0$ follows by Lemma~\ref{Lem:W12conv}. Consequently, $\frac{u}{h}\in T_h\core$ is approximated  in the $Q_h$-norm by $\{h'_n u\}_{n\in\N}\subset\core$. Since $T_h\core\subset\Dd(Q_h)$ is a core, $\core\subset\Dd(Q_h)$ is also a core of $Q_h$.

\medskip

(iii) Since $T_h:\big(\Dd(Q),\|\cdot\|_Q\big)\to \big(\Dd(Q_h),\|\cdot\|_{Q_h}\big)$ is a surjective isometry, $H$ and $H_h$ are unitarily equivalent implying $\sigma(H)=\sigma(H_h)$.

\medskip

(iv) It is straightforward to check that $H$ is critical if and only if $H_h$ is critical. Furthermore, $T_h h$ is a $H_h$-harmonic function with minimal growth at infinity and $T_h h=1$. Thus, $1$ is an (Agmon) ground state of $H_h$.
\end{proof}
%%%%%%%%%%%%%%%%%%

%%%%%%%%%%%%%%%%%%
\begin{lemma}\label{Lem:AdSeq-crit}
Let $Q$ be a semi-bounded form as defined in \eqref{Eq:ForSchr} with associated self-adjoint operator $H$. Let $W:\Omega\to \R$ be a bounded potential such that $H+W$ is non-negative. Let $h$ be a positive $(H+W)$-harmonic function. Then

\begin{itemize}

\item[(i)] $\{\varphi_n\}_{n\in\NM}$ is an admissible cut-off sequence for $(H,\varphi)$ if and only if $\{\varphi_n\}_{n\in\NM}$ is an admissible cut-off sequence for $(H_h,\varphi)$.

\item[(ii)] $\{\varphi_n\}_{n\in\NM}$ is an admissible cut-off sequence for $(H,\varphi)$ if and only if $\{\frac{\varphi_n}{h}\}_{n\in\NM}$ is an admissible cut-off sequence for $(H_h,\frac{\varphi}{h})$.

\end{itemize}
\end{lemma}
%%%%%%%%%%%%%%%%%%

%%%%%%%%%%%%%%%%%%
\begin{proof}
The constraints (i)-(ii) in Definition~\ref{Def:AdmSeq} are independent of the operator $H$. Thus, in order to show (i), it suffices to show that the weak Hardy inequality \eqref{Eq:WH} holds for $Q$ if and only if it holds for $Q_h$.  

\medskip

Let $v\in\mathcal{D}(Q)$. Since $T_h:\big(\Dd(Q),\|\cdot\|_Q\big)\to \big(\Dd(Q_h),\|\cdot\|_{Q_h}\big)$ is an isometry, $\|v\|_Q=\|T_hv\|_{Q_h}$. Also, notice that since $\mu=h^2\mathrm{d}m$,

$$\int_\Omega |T_h^{-1}v|^2\left|\nabla \left(\frac{\varphi_n}{\varphi}\right)\right|^2_A\,\mathrm{d}\mu=
\int_\Omega |v|^2\left|\nabla \left(\frac{\varphi_n}{\varphi}\right)\right|^2_A\,\mathrm{d}m.$$
Hence, the weak Hardy inequality \eqref{Eq:WH} holds for $Q$ if and only if it holds for $Q_h$. This proves (i). The statement in (ii) is a consequence of the general fact that if $\{\varphi_n\}_{n\in\NM}$ is an admissible cut-off sequence for $(H,\varphi)$, and if $h$ is any positive function in $C(\Omega)\cap W^{1,2}_{loc}(\Omega)$, then $\{\frac{\varphi_n}{h}\}_{n\in\NM}$ is an admissible cut-off sequence for $(H_h,\frac{\varphi}{h})$.
\end{proof}
%%%%%%%%%%%%%%%%%%

%%%%%%%%%%%%%%%%%%%%%%%%%%%%%%%%%%%%%%%%%%%%%%%%%%%%%%%%%%%%%%%%%%%%%%%%%%%%%%%%%%%%%
%%%%%%%%%%%%%%%%%%%%%%%%%%%%%%%%%%%%%%%%%%%%%%%%%%%%%%%%%%%%%%%%%%%%%%%%%%%%%%%%%%%%%
\section{Caccioppoli-type inequalities}
\label{Sec:Cacc}
%%%%%%%%%%%%%%%%%%%%%%%%%%%%%%%%%%%%%%%%%%%%%%%%%%%%%%%%%%%%%%%%%%%%%%%%%%%%%%%%%%%%%
%%%%%%%%%%%%%%%%%%%%%%%%%%%%%%%%%%%%%%%%%%%%%%%%%%%%%%%%%%%%%%%%%%%%%%%%%%%%%%%%%%%%%

This section is devoted to proving some Caccioppoli(-type) estimate, that will be one of the key ingredients in the proof of the main theorem. As has been seen in Section 2, Proposition \ref{Prop:GrouStaTransf}, we will have to consider the quadratic form

\begin{equation}\label{Eq:a}
\ab(v,w)
	:=\int_\Omega \langle A\nabla v,\nabla w \rangle\,d\mu
\end{equation}
The sesquilinear form $\ab$ is considered on $\Dd(\ab):=\overline{\core}^{\|\cdot\|_{\ab}}\subseteq L^2(\Omega,\dmu)$ where $A$ satisfies \eqref{Eq:A}. In Proposition \ref{Prop:GrouStaTransf}, the measure $\mu$ is given by $\mu=h^2\mathrm{d}m$, however in this section $\mu$ will denote an arbitrary measure that is absolutely continuous with respect to $\mathrm{d}x$. The self-adjoint operator associated to $\ab$ is denoted by $L$. For a bounded $W:\Gw\to\RM$, the operator $L+W$ is studied in this section. Denote by $q$ its associated sesquilinear form, namely $q(v,w):=\ab(v,w)+\langle Wv,w\rangle_{2,\mu}$. Since $W$ is real-valued and uniformly bounded, $q$ is symmetric and semibounded, i.e., $q(v)\geq -\|W\|_\infty \|v\|_{2,\mu}^2$. Then $\|v\|_q:=\left(q(v)+(1+\|W\|_\infty)\|v\|_{2,\mu}^2\right)^{1/2}$ is the corresponding $q$-norm, c.f. Section~\ref{Sec:Prel}.

\medskip

It is clear that the constant function equals to $1$ satisfies $L1\equiv0$. Throughout this section, $\{\varphi_n\}_{n\in\NM}$ denotes an admissible cut-off sequence for $(L,1)$ according to Definition~\ref{Def:AdmSeq}. Denote $A_n:=\supp\big(\varphi_{n+1}(1-\varphi_{n-1})\big)$. 

%%%%%%%%%%%%%%%%%%
\begin{lemma}\label{Lem:AdmSeq}
The following assertions hold for all $n\in\NM$:
\begin{itemize}
\item[(a)] $\varphi_{n+1}(1-\varphi_{n-1})\equiv 1$ on $\supp(\nabla \varphi_n)$,

\item[(b)] $\supp(\nabla \varphi_{n-1})\cap \supp(\nabla \varphi_{n+1})=\emptyset$.

\end{itemize}
\end{lemma}
%%%%%%%%%%%%%%%%%%

%%%%%%%%%%%%%%%%%%
\begin{proof}
This is straightforward and follows from (ii) in Definition \ref{Def:AdmSeq}.
\end{proof}
%%%%%%%%%%%%%%%%%%

\medskip

%%%%%%%%%%%%%%%%%%
\begin{lemma}\label{Lem:AdSeq-a-q}
Let $\ab$ be the form defined in \eqref{Eq:a} with associated self-adjoint operator $L$. For a bounded $W:\Gw\to\RM$, consider the operator $L+W$ with form $q$ defined by $q(v,w):=\ab(v,w)+\langle Wv,w\rangle_{2,\mu}$. Then
%%%%%%%%%%%%%%%%%%
$$
\|v\|_\ab\leq\|v\|_q\leq \sqrt{(2+\|W\|_\infty)} \,\|v\|_\ab
$$ 
In particular, a sequence $\{\varphi_n\}_{n\in\mathbb{N}}$ is admissible for $(L,\varphi)$ if and only if it is admissible for $(L+W,\varphi)$.
%%%%%%%%%%%%%%%%%%

\end{lemma}
%%%%%%%%%%%%%%%%%%

%%%%%%%%%%%%%%%%%%
\begin{proof}
Since $W$ is bounded, the desired estimate between $\|v\|_\ab$ and $\|v\|_q$ follows by a short computation. 
\end{proof}
%%%%%%%%%%%%%%%%%%

\medskip

For $a,b\in\RM$, we write $a\lesssim b$, if there is a constant $C>0$ such that $a\leq C\, b$. The following statement is a generalization of  \cite[Proposition~4.1]{BP17}.

%%%%%%%%%%%%%%%%%%
\begin{proposition}[Caccioppoli-type inequality I] 
\label{Prop:PointCaccio}
Let $W:\Gw\to\RM$ be bounded, and recall that $q$ is the quadratic form associated to $L+W$. Let $\{\varphi_n\}_{n\in\mathbb{N}}$ be an admissible cut-off sequence for $(L,1)$. Consider a generalized eigenfunction $u\in W^{1,2}_\loc(\Omega)$ of the operator $L+W$ with eigenvalue $\lambda\in\RM$. Then,
%%%%%%%%%%%%%%%%%%
$$
\int_{\Omega} |\varphi_{n+1}(1-\varphi_{n-1})|^2 |v|^2  |\nabla u|_A^2 \dmu \;
	\lesssim \; (2 + \sqrt{|\lambda| + \|W\|_\infty})^2\max_{x\in A_n}|u|^2
$$
%%%%%%%%%%%%%%%%%%
holds for every $v\in\Dd(q)$ satisfying $\|v\|_q\leq 1$.
\end{proposition}
%%%%%%%%%%%%%%%%%%

%%%%%%%%%%%%%%%%%%
\begin{proof}
Let $v\in\core$ with $\|v\|_q\leq 1$. Define $\tilde{v}:= \varphi_{n+1}(1-\varphi_{n-1}) v$ and
%%%%%%%%%%%%%%%%%%
$$
z \;
	:= \; \sqrt{\int_\Omega |\tilde{v}|^2  |\nabla u|_A^2 \dmu}
	\,.
$$
%%%%%%%%%%%%%%%%%%
The constraint $\|v\|_q\leq 1$ implies $\|v\|_{2,\mu}\leq 1$ and $\ab(v,v)\leq 1$ since $\langle W \, v,v\rangle_{2,\mu} + \|W\|_\infty \|v\|_{2,\mu}^2\geq 0$. Thus, a short computation invoking the Cauchy-Schwarz inequality and the fact that $0\leq \varphi_n\leq 1$ yields
%%%%%%%%%%%%%%%%%%
\begin{align*}
z^2 \;
	&= \; \left|
			\int_{\Omega} \big\langle A\nabla \big(\overline{\tilde{v}}\, \tilde{v} \, u\big),\nabla u \big\rangle \dmu
			-
			\int_{\Omega} u \big\langle A\nabla \big(\overline{\tilde{v}}\, \tilde{v}\big),\nabla u \big\rangle \dmu
		\right|\\
	&\leq \;
		\bigg|
			\int_{\Omega}
				(\lambda  - W) \, |\tilde{v}|^2  |u|^2
			\dmu
		\bigg|
		+
		\left|
			\int_{\Omega}
				u\, \overline{\tilde{v}}  \langle A\nabla \tilde{v}, \nabla u\rangle
			\dmu
		\right|
		+
		\left|
			\int_{\Omega}
				u\, \tilde{v}  \langle A\nabla \overline{\tilde{v}}, \nabla u\rangle
			\dmu
		\right|\\
	&\leq \; (|\lambda| + \|W\|_\infty) \, 
			\int_{\Omega}
				|\tilde{v}|^2  |u|^2
			\dmu
		+
		2\, 
		\left(
			\int_{\Omega}
				|u|^2  |\nabla \tilde{v}|_A^2
			\dmu
		\right)^{\frac{1}{2}}
		\left(
			\int_{\Omega}
				|\tilde{v}|^2  |\nabla u|_A^2
			\dmu
		\right)^{\frac{1}{2}}\\
	&\leq (|\lambda| + \|W\|_\infty) \max_{A_n}|u|^2 \int_\Omega|v|^2\dmu
				+
				  2z\,\, 
		\left(
			\int_{\Omega}
				|u|^2  |\nabla \tilde{v}|_A^2
			\dmu
		\right)^{\frac{1}{2}}\\
	&\leq  \max_{A_n}|u|^2 (|\lambda| + \|W\|_\infty)
			+
			  2z\,\, \max_{A_n}|u|
		\left(
			\int_{\Omega}
			  |\nabla \tilde{v}|_A^2
			\dmu
		\right)^{\frac{1}{2}}\,.
\end{align*}
%%%%%%%%%%%%%%%%%%
Since
%%%%%%%%%%%%%%%%%%
$$\nabla \tilde{v}
	=(1-\varphi_{n-1})v\nabla \varphi_{n+1}-\varphi_{n+1}v\nabla \varphi_{n-1}+\varphi_{n+1}(1-\varphi_{n-1})\nabla v
$$
%%%%%%%%%%%%%%%%%%
and $0\leq \varphi_k\leq 1$, the estimate
%%%%%%%%%%%%%%%%%%
\begin{align*}
\left(\int_{\Omega} |\nabla \tilde{v}|_A^2 \dmu\right)^{1/2}
	\leq & \left(\int_\Omega |v|^2|\nabla \varphi_{n-1}|^2_A\dmu\right)^{1/2}+\left(\int_\Omega |v|^2|\nabla \varphi_{n+1}|^2_A\dmu\right)^{1/2}+\left(\int_\Omega |\nabla v|_A^2\dmu\right)^{1/2}\\
	\leq &  \left(\int_\Omega |v|^2|\nabla \varphi_{n-1}|^2_A\dmu\right)^{1/2}+\left(\int_\Omega |v|^2|\nabla \varphi_{n+1}|^2_A\dmu\right)^{1/2}+1
\end{align*}
%%%%%%%%%%%%%%%%%%
follows. By assumption, $\{\varphi_n\}_{n\in\NM}$ is an admissible cut-off sequence of $L$. Therefore the weak Hardy inequality leads to

$$\left(\int_\Omega |v|^2|\nabla \varphi_{n-1}|^2_A\dmu\right)^{1/2}+\left(\int_\Omega |v|^2|\nabla \varphi_{n+1}|^2_A\dmu\right)^{1/2}\lesssim  ||v||_{\ab}.$$
By Lemma \ref{Lem:AdSeq-a-q} and using that $||v||_q\leq 1$, we thus obtain the existence of a constant $C\geq 1$ (independent of $v$) satisfying
%%%%%%%%%%%%%%%%%%
$$\left(\int_{\Omega} |\nabla \tilde{v}|_A^2 \dmu\right)^{1/2}
	\leq C\,.
$$
%%%%%%%%%%%%%%%%%%
Therefore, 
%%%%%%%%%%%%%%%%%%
$$z^2\leq \max_{A_n}|u|^2 (|\lambda| + \|W\|_\infty) 
			+
			  2C \max_{A_n}|u|\, z\,.
$$
%%%%%%%%%%%%%%%%%%
A straightforward study of this quadratic inequality implies that
%%%%%%%%%%%%%%%%%%
$$z
	\;\leq\; (2C + \sqrt{|\lambda| + \|W\|_\infty})  \max_{A_n} |u|
	\;\lesssim\; (2 + \sqrt{|\lambda| + \|W\|_\infty})  \max_{A_n} |u|
$$
%%%%%%%%%%%%%%%%%%
proving the desired estimate for all $v\in\core$ with $\|v\|_q\leq 1$.

\vspace{.1cm}

Let $v\in\Dd(q)$ with $\|v\|_q\leq 1$. Since $\core$ is a core in the domain $\Dd(q)$, there is a sequence $v_n\in\core$ with $\|v_n\|_q\leq 1$ that converges to $v$ in the $q$-norm. Hence, $\{v_n\}_{n\in\NM}$ converges to $v$ in the $L^2$-norm and so there is no loss of generality in assuming that it converges $\dmu$-a.e. to $v$.  Combined with Fatou's Lemma, this yields
%%%%%%%%%%%%%%%%%%
$$
\int_{\Omega} |\varphi_{n+1}(1-\varphi_{n-1}) v|^2  |\nabla u|_A^2 \dmu \;
	\leq \liminf_{n\to\infty} \int_{\Omega} |\tilde{v}_n|^2  |\nabla u|_A^2 \dmu
	\lesssim \; (2+|\lambda|+\|W\|_\infty)^2 \max_{x\in A_n}|u|^2\,,
$$
%%%%%%%%%%%%%%%%%%
finishing the proof.
\end{proof}
%%%%%%%%%%%%%%%%%%

\medskip

A result analogous to the previous statement involving the $L^2$-norm of $u$ and not the pointwise one can be obtained with slightly different estimates.

%%%%%%%%%%%%%%%%%%
\begin{proposition}[Caccioppoli-type inequality II] \label{Prop:L2Caccio}

Let $W:\Gw\to\RM$ be bounded, and recall that $q$ is the quadratic form associated to $L+W$. Let $\{\varphi_n\}_{n\in\mathbb{N}}$ be an admissible cut-off sequence for $(L,1)$. Consider a generalized eigenfunction $u\in W^{1,2}_\loc(\Omega)$ of the operator $L+W$ with eigenvalue $\lambda\in\RM$. Then,
%%%%%%%%%%%%%%%%%%
$$
\int_{\Omega} |\varphi_{n+1}(1-\varphi_{n-1})|^2  |\nabla u|_A^2 \dmu \;
	\lesssim \; \|u\|_{L^2(A_n,\mu)}+
	\left(\int_\Omega |u|^2|\nabla \varphi_{n-1}|_A^2\dmu
		+\int_\Omega |u|^2|\nabla \varphi_{n+1}|_A^2\dmu
	\right)^{1/2}
$$
holds for all $n\in\NM$ where the constant in the estimate depends on $\|W\|_\infty$ and $\lambda$.

\end{proposition}
%%%%%%%%%%%%%%%%%%

%%%%%%%%%%%%%%%%%%
\begin{proof}
Let $\psi:=\varphi_{n+1}(1-\varphi_{n-1})\in\core$ and $z:=\sqrt{\int_{\Omega} |\psi|^2  |\nabla u|_A^2 \dmu}$. A similar computation as in the proof of Proposition~\ref{Prop:PointCaccio} leads to
%%%%%%%%%%%%%%%%%%
\begin{align*}
z^2 &
	\leq (|\lambda|+\|W\|_\infty)\int_\Omega |\psi|^2 |u|^2\dmu
		+ 2z \left(\int_\Omega |u|^2 |\nabla\psi|_A^2\dmu \right)^{1/2} \\
	&\lesssim \|u\|_{L^2(A_n,\mu)}^2+2z\left(\int_\Omega |u|^2 |\nabla\psi|_A^2\dmu \right)^{1/2}\,,
\end{align*}
%%%%%%%%%%%%%%%%%%
where the constant in the latter estimate depends on $\lambda$ and $\|W\|_\infty$.
Since $\{\varphi_k\}_{k\in\NM}$ is an admissible sequence for $(L,1)$, $\nabla\varphi_{n+1}$ and $\nabla \varphi_{n-1}$ have disjoint support by Lemma~\ref{Lem:AdmSeq}. Thus, $\nabla \psi=(1-\varphi_{n-1})\nabla \varphi_{n+1}-\varphi_{n+1}\nabla \varphi_{n-1}$ and $0\leq \varphi_k\leq 1$ yield
%%%%%%%%%%%%%%%%%%
$$|\nabla \psi|_A^2
	\leq  |\nabla \varphi_{n-1}|_A^2+|\nabla \varphi_{n+1}|_A^2\,.
$$
%%%%%%%%%%%%%%%%%%
Hence, we arrive to the quadratic inequality
%%%%%%%%%%%%%%%%%%
$$z^2
	\lesssim \|u\|_{L^2(A_n,\mu)}^2
		+ 2z 
		\left(\int_\Omega |u|^2|\nabla \varphi_{n-1}|_A^2\dmu
			+ \int_\Omega |u|^2|\nabla \varphi_{n+1}|_A^2\dmu
		\right)^{1/2}\,.
$$
%%%%%%%%%%%%%%%%%%
This leads immediately to the desired estimate.
\end{proof}
%%%%%%%%%%%%%%%%%%

%%%%%%%%%%%%%%%%%%%%%%%%%%%%%%%%%%%%%%%%%%%%%%%%%%%%%%%%%%%%%%%%%%%%%%%%%%%%%%%%%%%%%
%%%%%%%%%%%%%%%%%%%%%%%%%%%%%%%%%%%%%%%%%%%%%%%%%%%%%%%%%%%%%%%%%%%%%%%%%%%%%%%%%%%%%
\section{Proof of the main result}
\label{Sec:PfMain}
%%%%%%%%%%%%%%%%%%%%%%%%%%%%%%%%%%%%%%%%%%%%%%%%%%%%%%%%%%%%%%%%%%%%%%%%%%%%%%%%%%%%%
%%%%%%%%%%%%%%%%%%%%%%%%%%%%%%%%%%%%%%%%%%%%%%%%%%%%%%%%%%%%%%%%%%%%%%%%%%%%%%%%%%%%%

Similarly to Section~\ref{Sec:Cacc}, given a measure $\mu$ on $\Omega$, which is absolutely continuous with respect to $dx$, the sesquilinear form 
%%%%%%%%%%%%%%%%%%
$$\ab(v,w)
	:=\int_\Omega \langle A\nabla v,\nabla w \rangle\,d\mu
$$
%%%%%%%%%%%%%%%%%%
is considered on $\Dd(\ab):=\overline{\core}^{\|\cdot\|_{\ab}}$. The associated self-adjoint operator is denoted by $L$. Furthermore, $q$ denotes the sesquilinear form of the operator $L+W$ for a bounded $W:\Gw\to\RM$. 

%%%%%%%%%%%%%%%%%%
\begin{lemma}\label{Lem:Main(i)}
Let $u\in W^{1,2}_{\loc}(\Omega)$ be a generalized eigenfunction of the operator $L+W$ with eigenvalue $\lambda\in \RM$. Let $\{\varphi_n\}_{n\in\mathbb{N}}$ be an admissible cut-off sequence for $(L,1)$. Recall that $A_n=\mathrm{supp}(\varphi_{n+1}(1-\varphi_{n-1}))$. If 
%%%%%%%%%%%%%%%%%%
$$\liminf_{n\to\infty} \frac{\max_{A_n}\left|u\right|}{\|\varphi_n u\|_{2,\mu}}\;\mathbf{a}(\varphi_n,\varphi_n)^{1/2}=0\,,
$$
%%%%%%%%%%%%%%%%%%
then  $\lambda\in\sigma(L+W)$.
\end{lemma}
%%%%%%%%%%%%%%%%%%

%%%%%%%%%%%%%%%%%%
\begin{proof}
%\red{THERE IS NO NEED FOR THE FOLLOWING PART (I JUST NEEDED IT WITH YEHUDA FOR THE POINTWISE ESTIMATES}
%\blue{Without loss of generality $u\not\in L^2(\Omega,\dmu)$. Thus, $\lim_{n\to\infty}\|\varphi_n u\|_{2,\mu}=\infty$ follows (by Fatou's Lemma, since $\{\varphi_n\}_{n\in\NM}$ converges pointwise to one). Thus, we can assume without loss of generality that $\|\varphi_n u\|_{2,\mu}\geq 1$.}
Define $w_n:=\frac{\varphi_n u}{\|\varphi_n u\|_{2,\mu}}$ for $n\in\NM$. Let $v\in\Dd(q)$ be so that $\|v\|_q\leq 1$. Since $u$ is a generalized eigenfunction of $L+W$, we have
%%%%%%%%%%%%%%%%%%
$$
\lambda  \langle w_n,v\rangle
    = \frac{\lambda}{\|\varphi_n u\|_{2,\mu}}  \langle u,\varphi_n v\rangle
	=  \frac{1}{\|\varphi_n u\|_{2,\mu}} \, q(u,\varphi_n v) \;
	=  \frac{1}{\|\varphi_n u\|_{2,\mu}} \, \ab(u,\varphi_n v) + \langle Ww_n,v\rangle\,.
$$
%%%%%%%%%%%%%%%%%%
Hence,  $\big|q(w_n,v) - \lambda  \langle w_n,v\rangle\big| = \frac{1}{\|\varphi_n u\|_{2,\mu}} \big|\ab(\varphi_n u,v) - \ab(u,\varphi_n v)\big|$. Therefore, the Leibniz rule implies
%%%%%%%%%%%%%%%%%%
$$
\Big|
	q(w_n,v) - \lambda  \langle w_n,v\rangle
\Big|\;
	=\; \frac{1}{\|\varphi_n u\|_{2,\mu}}
		\left|
			\int_{\Omega} u \langle A\nabla\varphi_n, \nabla v \rangle \dmu
			-
			\int_{\Omega} \bar{v} \langle A\nabla u, \nabla \varphi_n \rangle \dmu
		\right|\,.
$$
%%%%%%%%%%%%%%%%%%
Note that
%%%%%%%%%%%%%%%%%%
$$
\ab(v,v)+||v||_{2,\mu}^2\leq ||v||^2_q\leq 1.
$$
%%%%%%%%%%%%%%%%%%
Therefore, the above quantity is estimated by
%%%%%%%%%%%%%%%%%%
\begin{align*}
&\frac{1}{\|\varphi_n u\|_{2,\mu}} \left[
	\left(
		\int_{\Omega}
			|u|^2 |\nabla\varphi_n|_A^2
		\dmu
	\right)^{\frac{1}{2}}
	\ab(v,v)^{\frac{1}{2}}
	\!\!+
	\!\left(
		\int_{\mathrm{supp}(\nabla \varphi_n)}\!\!\!
			|v|^2 |\nabla u|_A^2
		\dmu
	\right)^{\frac{1}{2}}
	\ab(\varphi_n,\varphi_n)^{\frac{1}{2}}
\right]\\[2mm]
\leq & \frac{1}{\|\varphi_n u\|_{2,\mu}}  \left[
	\!\!\left(
		\int_{\Omega}
			|u|^2 |\nabla\varphi_n|_A^2
		\dmu\!\!
	\right)^{\frac{1}{2}}
	\!+
	\!\!\left(\!
		\int_{\Omega}
			|\varphi_{n+1}(1-\varphi_{n-1})|^2\, |v|^2  |\nabla u|_A^2
		\dmu
	\right)^{\frac{1}{2}}
\!\!\!	\ab(\varphi_n,\varphi_n)^{\frac{1}{2}}
\!\right]\,,
\end{align*}
%%%%%%%%%%%%%%%%%%
since $\varphi_{n+1}(1-\varphi_{n-1})=1$ on $\mathrm{supp}(\nabla \varphi_n)$ by Lemma~\ref{Lem:AdmSeq}. Moreover, this yields $\supp(\nabla\varphi_n)\subseteq A_n$ implying
%%%%%%%%%%%%%%%%%%
$$\int_{\Omega}
			|u|^2 |\nabla\varphi_n|_A^2
		\dmu
	\leq \max_{A_n}|u|^2\; \mathbf{a}(\varphi_n,\varphi_n)\,.
$$
%%%%%%%%%%%%%%%%%%
According to Proposition~\ref{Prop:PointCaccio}, 
%%%%%%%%%%%%%%%%%%
$$\int_{\Omega}
			|\varphi_{n+1}(1-\varphi_{n-1})|^2\, |v|^2  |\nabla u|_A^2
		\dmu 
		\lesssim (2+\sqrt{|\lambda|+\|W\|_\infty})^2\; \max_{A_n}|u|^2
$$
%%%%%%%%%%%%%%%%%%
follows for all $v\in\Dd(q)$ with $\|v\|_q\leq 1$. Hence, the previous considerations lead to
%%%%%%%%%%%%%%%%%%
$$\Big|
	q(w_n,v) - \lambda  \langle w_n,v\rangle
\Big|\; 
	\lesssim \frac{\max_{A_n}|u|}{\|\varphi_n u\|_{2,\mu}}\;\mathbf{a}(\varphi_n,\varphi_n)^{1/2}\,,
$$
%%%%%%%%%%%%%%%%%%
where the constant in the estimate depends on $\lambda$ and $\|W\|_\infty$. Using the hypothesis, one concludes that
%%%%%%%%%%%%%%%%%%
$$\liminf_{n\to\infty}\sup_{\|v\|_q\leq 1}\Big|
	q(w_n,v) - \lambda  \langle w_n,v\rangle
\Big|=0\,,
$$
%%%%%%%%%%%%%%%%%%
implying $\lambda\in\sigma(H)$ by Proposition~\ref{Prop:CritWeyl}.
\end{proof}
%%%%%%%%%%%%%%%%%%

%%%%%%%%%%%%%%%%%%
\begin{lemma}\label{Lem:Main(ii)}
Let $u\in W^{1,2}_{\loc}(\Omega)$ be a generalized eigenfunction of the operator $L+W$ with eigenvalue $\lambda\in \RM$. Let $\{\varphi_n\}_{n\in\mathbb{N}}$ be an admissible cut-off sequence for $(L,1)$. Recall that $A_n=\mathrm{supp}(\varphi_{n+1}(1-\varphi_{n-1}))$. If 
%%%%%%%%%%%%%%%%%%
$$\liminf_{n\to\infty}\frac{\|u\|_{L^2(A_n,\mu)}+\sqrt{\int_\Omega |u|^2(|\nabla \varphi_{n-1}|^2_A+|\nabla \varphi_{n}|_A^2+|\nabla \varphi_{n+1}|_A^2)\,\dmu}}{\|\varphi_n u\|_{2,\mu}}=0\,,
$$
%%%%%%%%%%%%%%%%%%
then  $\lambda\in\sigma(L+W)$.
\end{lemma}
%%%%%%%%%%%%%%%%%%

%%%%%%%%%%%%%%%%%%
\begin{proof}
%Up to extracting a subsequence, one can and will assume in the assumption that the limit, and not only the liminf, is equal to zero. 
Define the sequence
%%%%%%%%%%%%%%%%%%
$$
a_n:=\frac{\|u\|_{L^2(A_n,\mu)}+\sqrt{\int_\Omega |u|^2(|\nabla \varphi_{n-1}|^2_A+|\nabla \varphi_{n}|_A^2+|\nabla \varphi_{n+1}|_A^2)\,\dmu}}{\|\varphi_n u\|_{2,\mu}}\,.
$$
%%%%%%%%%%%%%%%%%%
%Let $v\in\mathcal{D}(q)$ such that $\|v\|_q\leq 1$. 
Following the computations in the proof of Lemma~\ref{Lem:Main(i)}, we get
%%%%%%%%%%%%%%%%%%
$$
\sup_{\|v\|_q\leq 1}
\Big|
	q(w_n,v) - \lambda  \langle w_n,v\rangle_{2,\mu}
\Big|\;
	\leq \; 
	\sup_{\|v\|_q\leq 1}
	\frac{		
		\left(
			\int_\Omega |u|^2 |\nabla\varphi_n|_A^2\dmu
		\right)^{1/2}
	}{\|\varphi_n u\|_{2,\mu}} +
	\sup_{\|v\|_q\leq 1}
	\frac{		
		\left|
			\int_{\Omega} \bar{v} \langle A\nabla u, \nabla \varphi_n \rangle\dmu
		\right|
	}{\|\varphi_n u\|_{2,\mu}}
$$
%%%%%%%%%%%%%%%%%%
We show in the following that the right hand side is bounded from above by $Ca_n$ for some constant $C>0$. Hence, $\lambda\in\sigma(L+W)$ follows from Proposition~\ref{Prop:CritWeyl} as $\liminf_{n\to\infty}a_n=0$ by assumption.

\medskip

Clearly, the first quotient is bounded from above by $a_n$, hence it is enough to treat the second one. Denote $\tilde{v}:=\varphi_{n+1}^2(1-\varphi_{n-1})^2v$, and note that $\tilde{v}=v$ on $\mathrm{supp}(\nabla\varphi_n)$ implying 
%%%%%%%%%%%%%%%%%%
$$\int_{\Omega} \bar{v} \langle A\nabla u, \nabla \varphi_n \rangle \dmu
	= \int_{\Omega} \bar{\tilde{v}}  \langle A\nabla u, \nabla \varphi_n \rangle\dmu\,.
$$
%%%%%%%%%%%%%%%%%%
Hence, using integration by parts, and that $u$ is a generalized eigenfunction of $L+W$ for the eigenvalue $\lambda$, one gets
%%%%%%%%%%%%%%%%%%
\begin{align*}
&\int_{\Omega} \bar{v} \langle A\nabla u, \nabla \varphi_n \rangle\dmu\\
= &\int_\Omega  (\lambda-W)\bar{v} \varphi_n \varphi_{n+1}^2(1-\varphi_{n-1})^2 u \dmu
	+ \int_\Omega \varphi_{n+1}^2(1-\varphi_{n-1})^2\varphi_n \langle A\nabla u,\nabla \bar{v}\rangle\dmu \\
&- 2\int_\Omega \varphi_{n+1}^2(1-\varphi_{n-1})\varphi_n \bar{v}\langle A\nabla u,\nabla \varphi_{n-1}\rangle\dmu
	+ 2\int_\Omega (1-\varphi_{n-1})^2\varphi_{n+1}\varphi_n \bar{v}\langle A\nabla u,\nabla \varphi_{n+1}\rangle\dmu\\
=: &(1)+(2)+(3)+(4)
\end{align*}
%%%%%%%%%%%%%%%%%%
In the following it is shown that each of the latter terms in absolute value and divided by $\|\varphi_n u\|_{2,\mu}$ is bounded from above by $a_n$ up to a multiple constant. Each of the summand (1),(2),(3) and (4) is treated separately.

\medskip

Note that $\|v\|_{2,\mu}\leq\|v\|_q\leq 1$. By Cauchy-Schwarz, and the fact that (by definition of $A_n$) the function $\varphi_{n+1}^2(1-\varphi_{n-1})^2$ has support in $A_n$, the absolute value of the term $(1)$ is bounded by
%%%%%%%%%%%%%%%%%%
$$(|\lambda|+\|W\|_\infty) \|u\|_{L^2(A_n,\mu)}\|v\|_{2,\mu}
	\lesssim \|u\|_{L^2(A_n,\mu)}
	\lesssim a_n\, \|\varphi_n u\|_{2,\mu}\,,
$$
%%%%%%%%%%%%%%%%%%
where the constant in the estimate is independent of $v$.
By Cauchy-Schwarz and $\|v\|_q\leq 1$, the absolute value of the term $(2)$ is bounded by
%%%%%%%%%%%%%%%%%%
$$\left(\int_\Omega \varphi_{n+1}^2(1-\varphi_{n-1})^2|\nabla u|_A^2\dmu\right)^{1/2}\,.
$$
%%%%%%%%%%%%%%%%%%
Notice that we use the fact that each of the functions $\varphi_n$, $\varphi_{n+1}$, $(1-\varphi_{n-1})$ is between $0$ and $1$ in order to eliminate the extra powers. According to Proposition~\ref{Prop:L2Caccio}, up to a multiplicative constant (independent of $v$ with $\|v\|_q\leq 1$), the above expression is bounded by
%%%%%%%%%%%%%%%%%%
$$\|u\|_{L^2(A_n,\mu)}
	+ \left(\int_\Omega |u|^2|\nabla \varphi_{n-1}|_A^2 \dmu +\int_\Omega |u|^2|\nabla \varphi_{n+1}|_A^2\dmu\right)^{1/2}
	\leq a_n\, \|\varphi_n u\|_{2,\mu}\,.
$$
%%%%%%%%%%%%%%%%%%
%Therefore, using the assumption, one gets that the limit as $n\to \infty$ of the quotient of $(2)$ and $\|\varphi_n u\|_{2,\mu}$ is zero.
Concerning $(3)$, using that $\varphi_n$ and $\varphi_{n+1}$ are between $0$ and $1$, we estimate its absolute value by
%%%%%%%%%%%%%%%%%%
$$2\int_\Omega \varphi_{n+1}(1-\varphi_{n-1})|v| |\nabla u|_A|\nabla \varphi_{n-1}|_A\dmu\,,
$$
%%%%%%%%%%%%%%%%%%
which is estimated by Cauchy Schwarz by
%%%%%%%%%%%%%%%%%%
$$\left(\int_\Omega \varphi_{n+1}^2(1-\varphi_{n-1})^2|\nabla u|_A^2\dmu\right)^{1/2}\left(\int_\Omega |v|^2|\nabla \varphi_{n-1}|_A^2\dmu\right)^{1/2}\,.
$$
%%%%%%%%%%%%%%%%%%
By Proposition~\ref{Prop:L2Caccio}, up to a multiplicative constant, the first term in the expression above is estimated by 
%%%%%%%%%%%%%%%%%%
$$\|u\|_{L^2(A_n,\mu)}+\left(\int_\Omega |u|^2|\nabla \varphi_{n-1}|_A^2\dmu+\int_\Omega |u|^2|\nabla \varphi_{n+1}|_A^2\dmu\right)^{1/2}\,.
$$
%%%%%%%%%%%%%%%%%%
Thanks to the weak Hardy inequality, which holds by assumption, the term $\int_\Omega |v|^2|\nabla \varphi_{n-1}|_A^2\dmu$ is uniformly bounded in $v\in\Dd(q)$ satisfying $\|v\|_q\leq 1$. Thus, the term (3) is bounded from above by $a_n\, \|\varphi_n u\|_{2,\mu}$ up to a multiple constant (independent of $v$). Similarly, the term (4) is treated, which finishes the proof. 
\end{proof}
%%%%%%%%%%%%%%%%%%

\medskip

%%%%%%%%%%%%%%%%%%
\textbf{Proof of Theorem~\ref{thm:main}.} Suppose (i) holds. We stress that here, the function $\varphi$ is a positive $(H+W)$-harmonic function, for some $W:\Omega\to \R$ bounded. We make a ground state transform with $h=\varphi$, and consider the operator $L=T_h(H+W)T_h^{-1}$ associated with the quadratic form

$$\ab(v,w)=\int_\Omega \langle A\nabla v,\nabla w\rangle \,\mathrm{d}\mu,$$
where $\mu=h^2\mathrm{d}m$. The function $T_hu$ is a generalized eigenfunction for $L-W$, associated with the eigenvalue $\lambda$. According to Proposition \ref{Prop:GrouStaTransf}, $\sigma(L-W)=\sigma(H)$. Hence it is enough to prove that $\lambda$ belongs to the spectrum of $L-W$. Since by assumption $\{\varphi_n\}_{n\in\mathbb{N}}$ is an admissible cut-off sequence for $(H,\varphi)$, by (ii) in Lemma \ref{Lem:AdSeq-crit}, the sequence $\{\frac{\varphi_n}{\varphi}\}_{n\in\mathbb{N}}$ is admissible for $(H_h,1)$, with $H_h=L-W$. According to Lemma \ref{Lem:AdSeq-a-q}, the sequence $\{\frac{\varphi_n}{\varphi}\}_{n\in\mathbb{N}}$ is also admissible for $(L,1)$. Moreover,

$$\left\|\frac{\varphi_n}{\varphi}\, u\right\|_{2,m}
	= \left\|\frac{\varphi_n}{\varphi}\, T_h u\right\|_{2,\mu}.$$
Therefore, applying Lemma \ref{Lem:Main(i)} to $L-W$ and $T_hu$, we get that $\lambda\in \sigma(L-W)$.

\medskip

Suppose now that (ii) holds. Without loss of generality, we assume that $\varphi\equiv 1$. Since $H$ is bounded from below, one can find $W:\Omega \to \R$ bounded such that $H+W$ is non-negative (for instance, $W$ is a large constant). By the Allegretto-Piepenbrnink theorem, there exists a positive function $h$ that is $H+W$-harmonic. We make a ground state transform with $h$, and consider the operator $L=T_h(H+W)T_h^{-1}$ associated with the quadratic form

$$\ab(v,w)=\int_\Omega \langle A\nabla v,\nabla w\rangle \,\mathrm{d}\mu,$$
where $\mu=h^2\mathrm{d}m$. The function $T_hu$ is a generalized eigenfunction for $L-W$, associated with the eigenvalue $\lambda$. According to Proposition \ref{Prop:GrouStaTransf}, $\sigma(L-W)=\sigma(H)$. Hence it is enough to prove that $\lambda$ belongs to the spectrum of $L-W$. Since by assumption $\{\varphi_n\}_{n\in\mathbb{N}}$ is an admissible cut-off sequence for $(H,1)$, by (i) in Lemma \ref{Lem:AdSeq-crit} it is also admissible for $(H_h,1)$, and by Lemma \ref{Lem:AdSeq-a-q} it is also admissible for $(H_h+W,1)$. Since $L=H_h+W$, we obtain that  $\{\varphi_n\}_{n\in\mathbb{N}}$ is an admissible cut-off sequence for $(L,1)$. Since

\begin{align*}
&\frac{\|u\|_{L^2(A_n,\dm)}
		+ \sqrt{\int_\Omega |u|^2(|\nabla \varphi_{n-1}|^2_A
			+|\nabla \varphi_{n}|_A^2
			+|\nabla \varphi_{n+1}|_A^2)\,\dm}}{\|\varphi_n u\|_{2,m}}\\
			=	&\frac{\|T_h u\|_{L^2(A_n,\dmu)}
		+ \sqrt{\int_\Omega |T_h u|^2(|\nabla \varphi_{n-1}|^2_A
			+|\nabla \varphi_{n}|_A^2
			+|\nabla \varphi_{n+1}|_A^2)\,\dmu}}{\|\varphi_n T_h u\|_{2,\mu}}
\end{align*}
the fact that $\lambda\in \sigma(L-W)$ follows from Lemma \ref{Lem:Main(ii)} with $L-W$ and $T_hu$.
$\hfill\Box$
%%%%%%%%%%%%%%%%%%
\vspace{.2cm}

%%%%%%%%%%%%%%%%%%%%%%%%%%%%%%%%%%%%%%%%%%%%%%%%%%%%%%%%%%%%%%%%%%%%%%%%%%%%%%%%%%%%%
%%%%%%%%%%%%%%%%%%%%%%%%%%%%%%%%%%%%%%%%%%%%%%%%%%%%%%%%%%%%%%%%%%%%%%%%%%%%%%%%%%%%%
\section{Applications in the critical case}
\label{Sec:Crit}
%%%%%%%%%%%%%%%%%%%%%%%%%%%%%%%%%%%%%%%%%%%%%%%%%%%%%%%%%%%%%%%%%%%%%%%%%%%%%%%%%%%%%
%%%%%%%%%%%%%%%%%%%%%%%%%%%%%%%%%%%%%%%%%%%%%%%%%%%%%%%%%%%%%%%%%%%%%%%%%%%%%%%%%%%%%

%%%%%%%%%%%%%%%%%%

In this section, we explain how Theorem~\ref{thm:main} generalizes \cite[Theorem 1.1]{BP17}. We assume that $H$ is of the form \eqref{Eq:Schro}, is non-negative and {\em critical}, and denote by $h>0$ the unique (up to multiplicative constant) ground state for $H$. As in Theorem~\ref{thm:main}, one defines the quadratic form

%%%%%%%%%%%%%%%%%%
$$
\ab(u,v) \;
	:= \; \int_\Omega \langle A\nabla u, \nabla v\rangle\ d\mu (x),
	\,,
$$
where $\dmu=h^2\dm$. The main result of this section is the following:

%%%%%%%%%%%%%%%%%%
\begin{theorem} \label{thm:good_critical}

Assume that $H$ is as above, and that $W:\Omega\to \R$ is bounded such that $H+W$ is critical with ground state $h$. Then, there exists a good cut-off sequence $\{\varphi_n\}_{n\in\mathbb{N}}$ for $(H,h)$, such that $\ab(\frac{\varphi_n}{h},\frac{\varphi_n}{h})\to 0$ as $n\to\infty$ and $\bigcup_{n\in\N}\supp(\varphi_n)=\Omega$.

\end{theorem}
%%%%%%%%%%%%%%%%%%

%%%%%%%%%%%%%%%%%%
\begin{corollary}[\cite{BP17}, Theorem~1.1]\label{Cor:BP}

Let $u$ be a generalized eigenfunction associated to $\lambda$ for $H$, and assume that $W:\Omega\to\R$ be bounded such that $H+W$ is critical with ground state $h$. Assume that $|u|\lesssim  h$. Then, $\lambda\in\sigma(H)$.

\end{corollary}
%%%%%%%%%%%%%%%%%%

%%%%%%%%%%%%%%%%%%
\begin{proof}
If $u\in L^2(\Omega,\mathrm{d}m)$, then obviously $\lambda$ belongs to the spectrum of $H$. Therefore, without loss of generality, one can assume that $u\notin L^2(\Omega,\mathrm{d}m)$. Let $\{\varphi_n\}_{n\in\mathbb{N}}$ be the good cut-off sequence for $(H,h)$ provided by Theorem \ref{thm:good_critical}. Using (i), (ii) from the definition of a good cut-off sequence and $\bigcup_{n\in\N}\supp(\varphi_n)=\Omega$, it is easily seen that

$$\lim_{n\to \infty} \left\|\frac{\varphi_n}{h}u\right\|_{2,m}
	=\infty.
$$
By assumption and Theorem~\ref{thm:good_critical}, one has

$$\lim_{n\to \infty}\max_{A_n}\left|\frac{u}{h}\right|\ab\left(\frac{\varphi_n}{h},\frac{\varphi_n}{h}\right)
	=0.
$$
Hence,

$$\lim_{n\to \infty}\frac{\max_{A_n}\left|\frac{u}{h}\right|\ab(\frac{\varphi_n}{h},\frac{\varphi_n}{h})}{\|\frac{\varphi_n}{h}u\|_{2,m}}=0,$$
and the result follows from (i) in Theorem~\ref{thm:main}.
\end{proof}
%%%%%%%%%%%%%%%%%%

\medskip

The remaining part of this section will be devoted to the proof of Theorem~\ref{thm:good_critical}.

\medskip
 
We denote by $L:=-\diver\! (A\nabla\cdot )$ on $L^2(\Gw,\dmu)$ the self-adjoint operator associated to the quadratic form $\ab$. By ground state transform, $L=h^{-1}(H+W)h$, therefore $L$ is critical with ground state $1$. %Furthermore, by ground state transform, $\mathbf{a}$ is the quadratic form of $L$. 
Hence, by Lemma \ref{Lem:AdSeq-a-q} and (ii) in Lemma \ref{Lem:AdSeq-crit}, a reformulation of Theorem \ref{thm:good_critical} is that there exists a good cut-off sequence $\{\varphi_n\}_{n\in\mathbb{N}}$ for $(L,1)$, such that $\{\varphi_n\}_{n\in\mathbb{N}}$ is a null sequence for $L$.

%%%%%%%%%%%%%%%%%%%%%%%%%%%%%%%%%%%%%%%%%%%%%%%%%%%%%%%%%%%%%%%%%%%%%%%%%%%%%%%%%%%%%
\subsection{Construction of the null-sequence}
\label{Ssec:Null}
%%%%%%%%%%%%%%%%%%%%%%%%%%%%%%%%%%%%%%%%%%%%%%%%%%%%%%%%%%%%%%%%%%%%%%%%%%%%%%%%%%%%%

In this paragraph, we define the sequence $\{\varphi_n\}_{n\in\mathbb{N}}$ and prove that it is a null sequence for $L$ and a good cut-off sequence for $(L,1)$. %In the next paragraph, it will be shown to be also a good cut-off sequence for $(L,1)$.

The construction of the sequence $\{\varphi_n\}_{n\in\mathbb{N}}$ is based on the existence of Evans potentials, which is explained next. We let $K\Subset \Omega$ be a compact subset, and let $\varphi$ be a locally H\"older continuous, $W^{1,2}_{loc}(\Omega\setminus K)$ function. We say that $\varphi$ is an {\em Evans potential} for $L$ outside of $K$ if $\varphi$ is a positive, $L$-harmonic function on $\Omega\setminus K$, such that

$$\lim_{x\to\infty}\varphi(x)=+\infty.$$

\begin{lemma}\label{Lem:Evans}

For every open, relatively compact subset $U\subset \Omega$, there exists an Evans potential $\varphi$ for $L$ outside of $\bar{U}$.

\end{lemma}

\begin{proof}
Let $W\geq0$ be a smooth, compactly supported potential, that is not identically zero, and which vanishes outside of $U$. Then the operator $P=L+W$ is subcritical. Hence $P$ has positive minimal Green functions $G(x,y)$. Let us fix a pole $y\in U$. Since $W$ is compactly supported in $U$, the Green function $G(x,y)$ is $L$-harmonic on $\Omega\setminus U$, with minimal growth at infinity. By the fact that $L$ is critical with ground state equal to $1$, one has

$$G(x,y) \asymp 1$$
as $x\to \infty$. According to \cite[Theorem 1]{Anc}, there exists $\varphi:\Omega \to \R$, positive, $P$-harmonic such that

$$\lim_{x\to \infty}\frac{G(x,y)}{\varphi(x)}=0,$$
consequently

$$\lim_{x\to \infty} \varphi(x)=+\infty.$$
Since $W$ vanishes outside of $U$, $\varphi$ is $L$-harmonic on $\Omega\setminus U$, therefore it is an Evans potential.
\end{proof}

\medskip

In the sequel, we will fix an Evans potential $\varphi$ outside of some compact set $K$. Without loss of generality, up to changing $K$, one can assume that $\varphi|_{\partial K}=\frac{1}{2}$, and that $\varphi>\frac{1}{2}$ on $\Omega\setminus K$. We extend $\varphi$ to $K$ by the constant $\frac{1}{2}$. For $s>\frac{1}{2}$, we let $\Omega_s=\{\frac{1}{2}\leq \varphi<s\}$, and $A(r,R)=\{r<\varphi<R\}$. For $\frac{1}{2}<r<R$, we define $\psi_{r,R}(x)$ by

$$\psi_{r,R}(x)= \begin{cases}
\frac{R-\varphi(x)}{R-r},\qquad &x\in \overline{A(r,R)},\\
1,\qquad & x\in \Omega_r,\\
0,\qquad &x\notin \Omega_R.
\end{cases}$$
Clearly, $L\psi_{r,R}=0$ weakly in $A(r,R)$, and $\psi_{r,R}$ is $W^{1,2}(\Omega)$, and H\"older continuous. Recall that if $F$ is a compact set and $U$ is open such that $F\subseteq U$, then the relative capacity $\mathrm{Cap}_{\mathbf{a}}(F,U)$ with respect to the quadratic form $\mathbf{a}$ is defined by

$$\mathrm{Cap}_{\mathbf{a}}(F,U):=\inf_{\psi\in \mathcal{L}(F,U)}\mathbf{a}(\psi,\psi),$$
where $\mathcal{L}(F,U)$ denotes the set of all Lipschitz functions $\psi$ such that $\psi|_F\geq  1$, and $\supp(\psi)\subseteq U$. Let us recall (see \cite{LSW}) that a function $u\in W_0^{1,2}(U)$ is said to satisfy $u\geq 1$ on $F$ {\em in the sense of }$W^{1,2}_0(U)$, if there exists a sequence of functions $\{u_n\}_{n\in \N}$ in $\mathcal{L}(F,U)$ such that $u_n\to u$ in $W^{1,2}_0(U)$. One defines analogously that $u\leq 1$ on $F$ in the sense of $W^{1,2}_0(U)$; we say that $u=1$ on $F$ in the sense of $W^{1,2}_0(U)$ if both $u\leq 1$ and $u\geq1$ in the sense of $W^{1,2}_0(U)$. Let $\mathcal{H}(F,U)$ be the set of all functions $u\in W_0^{1,2}(U)$, such that $u\geq 1$ on $F$ in the sense of $W_0^{1,2}(U)$. It is easy to see by using the local uniform ellipticity of $A$, that if $U$ is relatively compact in $\Omega$, then

$$\mathrm{Cap}_{\mathbf{a}}(F,U)=\inf_{\psi\in \mathcal{H}(F,U)}\mathbf{a}(\psi,\psi).$$
The capacity of a compact set $F$ is then defined as

$$\mathrm{Cap}_{\mathbf{a}}(F)
	:=\lim_{n\to \infty}\mathrm{Cap}_{\mathbf{a}}(F,U_n),
$$
for any exhaustion $\{U_n\}_{n\in\mathbb{N}}$ of $\Omega$ . Clearly, $\mathrm{Cap}_{\mathbf{a}}(F,U_n)$ is a non-increasing sequence. In addition, it is straightforward to see that this definition does not depend on the choice of the exhaustion $\{U_n\}_{n\in\mathbb{N}}$.

It is well-known --at least when $L$ is the Laplacian-- that $L$ is critical, if and only if $\mathrm{Cap}_{\mathbf{a}}(F)=0$, for some (any) compact set $F\subset \Omega$.  For the sake of completeness, a proof of the similar statement in our setting is provided. %and since this result will be very useful for us in the sequel, we write down a proof:

\begin{lemma}\label{lem:capacity}

The operator $L$ is critical, if and only if for some (any) compact set $F\subseteq \Omega$, $\mathrm{Cap}_{\mathbf{a}}(F)=0$.

\end{lemma}

\begin{proof}
Assume that $L$ is critical, and let $\{h_n\}_{n\in\N}$ be a (non-negative) null sequence for $\mathbf{a}$ converging locally uniformly to $h>0$ a null-state. Let $F\subseteq \Omega$ be compact, and assume by contradiction that $\mathrm{Cap}_{\mathbf{a}}(F)=c>0$. Then, for every non-negative function $u\in C_0^\infty(\Omega)$,

$$\mathbf{a}(u,u)\geq c\min_F(u).$$
Since $h_n\to h$ locally uniformly and $h$ is positive, the quantity $\min_F(h_n)$ is bounded from below by a positive constant, independently of $n$ large enough. Hence,

$$\mathbf{a}(h_n,h_n)>c'>0,\quad n\gg 1.$$
This contradicts the fact that $\mathbf{a}(h_n,h_n)\to 0$ since $\{h_n\}_{n\in\N}$ is a null sequence. Therefore, $\mathrm{Cap}_{\mathbf{a}}(F)=0$ for all compact $F$. 

Conversely, assume that $L$ is not critical and let $F\subseteq \Omega$. Since $\mathrm{Cap}_{\mathbf{a}}(F)$ is non-increasing with respect to the set $F$, in order to prove that $\mathrm{Cap}_{\mathbf{a}}(F)>0$, one can assume that $\mu(F)>0$. According to \cite[Theorem 1.4]{PiTi06}, there exists a positive, continuous function $\rho$ such that

$$\mathbf{a}(u,u)\geq \int_\Omega \rho u^2\,\mathrm{d}\mu ,\quad u\in C_0^\infty(\Omega).$$
Note that this means that $L$ has a weighted spectral gap. An approximation argument shows that this inequality also holds if $u$ is merely Lipschitz with compact support in $\Omega$. If now $F$ is a compact set in $\Omega$, this implies that

$$\mathbf{a}(u,u)\geq \mu(F)\min_F\rho \cdot \min_F u^2$$
for all $u$ Lipschitz with compact support. If $u\geq1$ on $F$, then one obtains

$$\mathbf{a}(u,u)\geq \mu(F)\min_F\rho,$$
which implies that

$$\mathrm{Cap}_{\mathbf{a}}(F)\geq \mu(F) \min_F\rho>0.$$

\end{proof}
 
 \begin{lemma}\label{lem:Cap_attained}
Let $\frac{1}{2}<r<R$. Then, 
$$
\mathrm{Cap}_{\mathbf{a}}(\bar{\Omega}_r,\Omega_R)
	=\mathbf{a}(\psi_{r,R},\psi_{r,R})\,.
$$
 \end{lemma}
 
 \begin{proof}
 Let $F:=\bar{\Omega}_r$, $U:=\Omega_R$. According to \cite[Theorem 4.1]{LSW}, the capacity $\mathrm{Cap}_{\mathbf{a}}(F,U)$ is attained by the unique function $u$ in $W_0^{1,2}(U)$, such that $u=1$ on $F$ in the $W_0^{1,2}(U)$-sense, and such that $Lu=0$ weakly in $U\setminus F=A(r,R)$. Since $L\psi_{r,R}=0$ in $A(r,R)$, it is enough to show that $\psi_{r,R}\in W_0^{1,2}(U)$, and $\psi_{r,R}=1$ on $F$ in the $W_0^{1,2}(U)$-sense. By using well-known mollifier arguments, one can approximate the Evans potential $\varphi$ both uniformly, and in $W^{1,2}$ norm, by Lipschitz (in fact, even smooth) functions $\{\varphi_n\}_{n\in \N}$ on $\bar{U}$. Let $\rho$ be a smooth function on $\bar{U}$ such that $\rho|_F=1$ and $\rho|_{\partial U}=-1$. Up to replacing $\varphi_n$ by $\varphi_n+c_n\rho$, where $c_n=\max_{\partial U}(\varphi_n-r)+\max_{\partial F}(R-\varphi_n)$ (note that $c_n\to 0$), one can assume that $\varphi_n|_{\partial U}\leq R$ and $\varphi_n|_{\partial F}\geq r$. Then, up to replacing $\varphi_n$ by $\varphi_n\wedge r \vee R$, one can assume that $\varphi_n|_{\partial U}= R$ and $\varphi_n|_{\partial F}= r$. Define then a sequence $\{\psi^n_{r,R}\}_{n\in\N}$ by
$$
\psi^n_{r,R}(x)=\begin{cases}
\frac{R-\varphi_n(x)}{R-r},\qquad &x\in \overline{A(r,R)},\\
1,\qquad &x\in \Omega_r,\\
0,\qquad &x\notin \Omega_R.
\end{cases}
$$
The function $\psi_{r,R}^n$ is Lipschitz and belongs to $W^{1,2}_0(U)$, and as $n\to \infty$, $\psi_{r,R}^n$ converges uniformly and in $W^{1,2}$-norm on $A(r,R)$ to $\psi_{r,R}$. Since $\psi_{r,R}^n|_{F}=1$, it follows that $\psi_{r,R}=1$ on $F$ in the $W_0^{1,2}(U)$-sense. This concludes the proof.
 \end{proof}
\medskip
 
By Lemma \ref{lem:capacity}, for any fixed $r>\frac{1}{2}$, 
$$
0=\mathrm{Cap}_{\mathbf{a}}(\bar{\Omega}_r)
	=\lim_{R\to\infty}\mathrm{Cap}_{\mathbf{a}}(\bar{\Omega}_r,\Omega_R)\,.
$$
Hence, by Lemma \ref{lem:Cap_attained}, for every $r>\frac{1}{2}$,

\begin{equation}\label{Eq:Cap}
\lim_{R\to\infty}\int_{A(r,R)}\langle A\nabla \psi_{r,R},\nabla \psi_{r,R}\rangle\dmu=0.
\end{equation}
We now define our null-sequence as follows:
$$
\varphi_n
	:=\psi_{r_{n+1},R_{n+1}},\quad n\in\mathbb{N}^*\,,
$$
where the sequences $\{r_n\}_{n\in\mathbb{N}^*}$ and $\{R_n\}_{n\in\mathbb{N}^*}$ are defined recursively in the following way: first take $r_1>\frac{1}{2}$, and then define  $(r_{n})_{n\in\mathbb{N}^*}$ and $(R_{n})_{n\in\mathbb{N}^*}$ such that the following conditions are satisfied:

\begin{itemize}

\item[(H1)] $r_{n+1}>R_n$.

\item[(H2)] $R_{n+1}>r_{n+1}$.

\item[(H3)] $R_{n}$ is large enough so that $\frac{1}{R_{n}-r_{n}}\leq \frac{2}{R_{n}}.$

\item[(H4)] $R_{n}$ is large enough so that
$$
\int_{A(r_n,R_n)}\langle A\nabla \psi_{r_{n},R_{n}},\nabla \psi_{r_{n},R_{n}}\rangle\dmu <\frac{1}{n}\,,
$$
(which is possible according to \eqref{Eq:Cap}). 

\end{itemize}
Then, clearly $\{\varphi_n\}_{n\in\mathbb{N}}$ is a null-sequence, $0\leq \varphi_n\leq 1$, and $\varphi_{n+1}(1-\varphi_{n-1})=1$ on $\supp(\nabla\varphi_n)$. Also, notice that $A(r_n,R_n)\cap K=\emptyset$ for every $n$.

%%%%%%%%%%%%%%%%%%%%%%%%%%%%%%%%%%%%%%%%%%%%%%%%%%%%%%%%%%%%%%%%%%%%%%%%%%%%%%%%%%%%%
\subsection{The null-sequence \texorpdfstring{$\{\varphi_n\}_{n\in\mathbb{N}}$}{varphin} is a good cut-off sequence}
\label{Ssec:AdNull}
%%%%%%%%%%%%%%%%%%%%%%%%%%%%%%%%%%%%%%%%%%%%%%%%%%%%%%%%%%%%%%%%%%%%%%%%%%%%%%%%%%%%%

%%%%%%%%%%%%%%%%%%
In this paragraph, as indicated by its title, we prove that the null sequence $\{\varphi_n\}_{n\in\NM}$ constructed in the previous paragraph is indeed a good cut-off sequence. This will conclude the proof of Theorem \ref{thm:good_critical}. Obviously, the sequence $\{\varphi_n\}_{n\in\NM}$ satisfies (i), (ii) of Definition \ref{Def:AdmSeq}. In addition, it follows by construction that $\bigcup_{n\in\N}\supp(\varphi_n)=\Omega$ as $\lim_{x\to\infty}\varphi(x)=+\infty$ holds for the Evans potential. Hence, the only non-trivial point to check is that (iii) holds, and this is the purpose of the next proposition:

\begin{proposition}\label{Prop:Hardy}

There exists a constant $C>0$ such that, for every $n\in\mathbb{N}$ and every function $w\in \mathcal{D}(\ab)$,

\begin{equation}\label{Eq:wH_crit}
\int_\Omega |w|^2|\nabla \varphi_n|_A^2 \dmu \leq C\|w \|_{\ab}^2.
\end{equation}

\end{proposition}
%%%%%%%%%%%%%%%%%%

%%%%%%%%%%%%%%%%%%
\begin{proof}
Recall that $\varphi$ is the Evans potential from Lemma \ref{Lem:Evans}. First notice that

$$|\nabla \varphi_n|_A^2=|R_{n+1}-r_{n+1}|^{-2}|\nabla \varphi|_A^2 \mathbf{1}_{A(r_{n+1},R_{n+1})}.$$
We claim that for some constant $C$ and every $n\in\mathbb{N}^*$,

\begin{equation}\label{Eq:Hardy3}
|R_n-r_n|^{-1}
	\leq \frac{C}{\varphi(x)},\qquad x\in A(r_n,R_n).
\end{equation}
Indeed, by definition of the null-sequence $\{\varphi_n\}_{n\in\mathbb{N}}$,

$$\frac{1}{R_n}
	\leq \frac{1}{\varphi(x)}\leq \frac{1}{r_n}
	,\qquad x\in A(r_n,R_n).
$$
Using the hypothesis (H3) on the sequence $\{R_n\}_{n\in\mathbb{N}^*}$, we see that \eqref{Eq:Hardy3} holds with $C=2$. Consequently,

$$|\nabla \varphi_n|_A^2 
	\leq C\frac{|\nabla \varphi|_A^2}{\varphi^2}\mathbf{1}_{A(r_{n+1},R_{n+1})}.
$$
Next, we recall the following universal Hardy inequality: for every $w\in C_0^\infty(\Omega\setminus K)$,

\begin{equation}\label{Eq:Hardy}
\int_\Omega |w|^2\frac{|\nabla \varphi|_A^2}{4\varphi^2}\,\dmu
	\leq \int_\Omega |\nabla w|_A^2\,\dmu,
\end{equation}
where $K$ is the compact set in the definition of the Evans potential $\varphi$.
Inequality \eqref{Eq:Hardy} follows from the fact that since  $\varphi$ is a positive $L$-harmonic function on $\Omega\setminus K$, the function $\varphi^{1/2}$ is a positive solution of $(L-\frac{|\nabla \varphi|_A^2}{4\varphi^2})u=0$ in $\Omega\setminus K$, together with the celebrated Agmon-Allegretto-Moss-Piepenbrink theorem (see \cite{DeFrPi14}, especially Lemma 5.1 therein). 

\medskip

Let $\chi$ be a smooth, compactly supported function which is equal to $1$ on $K$, and such that its support is included in $\Omega_{r_1}$. Let $w\in C_0^\infty(\Omega)$. Then, using that $A\in L_{\loc}^\infty$ and elementary manipulations, one has

$$\begin{array}{rcl}
\int_\Omega |w|^2|\nabla \varphi_n|_A^2\dmu &\lesssim &\int_{A(r_{n+1},R_{n+1})} |w|^2\frac{|\nabla \varphi|_A^2}{4\varphi^2}\,\dmu\\\\
&\lesssim&  \int_{A(r_{n+1},R_{n+1})} |\chi w|^2\frac{|\nabla \varphi|_A^2}{4\varphi^2}\,\dmu+\int_{A(r_{n+1},R_{n+1})} |(1-\chi)w|^2\frac{|\nabla \varphi|_A^2}{4\varphi^2}\,\dmu\\\\
&\lesssim& 0+\int_\Omega |\nabla ((1-\chi)w)|_A^2\,\dmu\\\\
&\lesssim& (\|w\|^2_{2,\mu}+ \int_\Omega |\nabla w|_A^2\,\dmu)\\\\
&\lesssim& \|w\|^2_{\ab},
\end{array}$$
where in the third line we have used that the support of $\chi$ is disjoint from $A(r_{n+1},R_{n+1})$, and in the fourth one we have used the Hardy inequality \eqref{Eq:Hardy} for $(1-\chi)w$. Thus, we have proved that for all $w\in C_0^\infty(\Omega)$, and for all $n\in \mathbb{N}$,

\begin{equation}\label{Eq:Hardy2}
\int_\Omega |w|^2|\nabla \varphi_n|_A^2\dmu\leq C\|w\|^2_{\ab}.
\end{equation}
Let now $w$ be an element of $\mathcal{D}(\ab)$, and let $\{w_k\}_{k\in\mathbb{N}}$ be a sequence of smooth, compactly supported functions such that $\{w_k\}_{k\in\mathbb{N}}$ converges to $w$ in the $\ab$-norm; in particular, it converges to $w$ in $L^2(\Omega,\dmu)$, hence almost everywhere, and by the Fatou lemma,

\begin{align*}
\int_\Omega |w|^2|\nabla \varphi_n|_A^2\dmu  \leq &\liminf_{k\to\infty} \int_\Omega  |w_k|^2|\nabla \varphi_n|_A^2\dmu \\
 \leq & \liminf_{k\to\infty} C  \|w_k\|^2_{\ab}\\
 \leq &C \|w\|^2_{\ab}.
\end{align*}
\end{proof}
%%%%%%%%%%%%%%%%%%

\textbf{Proof of Theorem~\ref{thm:good_critical}.} Denote $L:=-\diver\! (A\nabla\cdot )$ on $L^2(\Gw,\dmu)$. By Lemma \ref{Lem:AdSeq-a-q} and (ii) in Lemma \ref{Lem:AdSeq-crit}, a reformulation of Theorem \ref{thm:good_critical} is that there exists a good cut-off sequence $\{\varphi_n\}_{n\in\mathbb{N}}$ for $(L,1)$, such that $\{\varphi_n\}_{n\in\mathbb{N}}$ is a null sequence for $L$. In Section~\ref{Ssec:Null}, a null-sequence $\{\varphi_n\}_{n\in\NM}$ of $L$ was constructed and due to Proposition~\ref{Prop:Hardy}, this sequence is also a good cut-off sequence for $(L,1)$. This concludes the proof.
\hfill$\Box$

%%%%%%%%%%%%%%%%%%%%%%%%%%%%%%%%%%%%%%%%%%%%%%%%%%%%%%%%%%%%%%%%%%%%%%%%%%%%%%%%%%%%%
%%%%%%%%%%%%%%%%%%%%%%%%%%%%%%%%%%%%%%%%%%%%%%%%%%%%%%%%%%%%%%%%%%%%%%%%%%%%%%%%%%%%%
\section{Generalized eigenfunctions with subexponential growth}
\label{Sec:subexp}
%%%%%%%%%%%%%%%%%%%%%%%%%%%%%%%%%%%%%%%%%%%%%%%%%%%%%%%%%%%%%%%%%%%%%%%%%%%%%%%%%%%%%
%%%%%%%%%%%%%%%%%%%%%%%%%%%%%%%%%%%%%%%%%%%%%%%%%%%%%%%%%%%%%%%%%%%%%%%%%%%%%%%%%%%%%

Before we present a consequence of Theorem~\ref{thm:main}, we recall a definition from \cite{BoLeSt09}: a function $J:\mathbb{N}\to[0,+\infty)$ is called {\em subexponentially bounded} if for every $\alpha>0$, there exists a constant $C_\alpha$ such that for all $n\in\mathbb{N}$,

$$J(n)\leq C_\alpha \mathrm{e}^{\alpha n}.$$
We have the following consequence of Theorem \ref{thm:main}:

%%%%%%%%%%%%%%%%%%
\begin{corollary}\label{Cor:Schnol3}

Let $H$ be of the form \eqref{Eq:Schro}, and let $u\in W^{1,2}_{\loc}(\Omega)$ be a generalized eigenfunction of the operator $H$, associated with the eigenvalue $\lambda\in \R$. Let us assume that $\{\varphi_n\}_{n\in \N}$ is an admissible cut-off sequence for $(H,1)$. Assume that the function $J(n):=\|\varphi_n u\|_{2,m}$ is subexponentially bounded, and that
$$
\lim_{n\to\infty}\frac{\sqrt{\int_\Omega |u|^2(|\nabla \varphi_{n-1}|^2_A+|\nabla \varphi_{n}|_A^2+|\nabla \varphi_{n+1}|_A^2)\dm}}{\|\varphi_n u\|_{2,m}}=0.
$$
Then, $\lambda\in\sigma(H)$.

\end{corollary}
%%%%%%%%%%%%%%%%%%

%%%%%%%%%%%%%%%%%%
\begin{remark}
{\em
Under the assumptions that the function $J(n):=\|\varphi_n u\|_{2,m}$ is subexponentially bounded, and $\|\varphi_n u\|_{2,m}\asymp \|\varphi_{n+1} u\|_{2,m} \asymp\|\varphi_{n-1} u\|_{2,m}$ as $n\to \infty$ (which in practice is often satisfied), then the assumption $(i)$ with $\varphi\equiv 1$ in Theorem~\ref{thm:main} implies $(ii)$ in Theorem~\ref{thm:main}. In fact, as will be proved in the course of the proof of Corollary~\ref{Cor:Schnol3}, if $J(n)$ is subexponentially bounded then

$$\liminf_{n\to\infty}\frac{\|u\|_{L^2(A_n,m)}}{\|\varphi_n u\|_{2,m}}=0.$$
However, in general it is not possible to compare the two assumptions $(i)$ and $(ii)$ in Theorem~\ref{thm:main}.
}
\end{remark}
%%%%%%%%%%%%%%%%%%

%%%%%%%%%%%%%%%%%%
\begin{proof}
According to Theorem~\ref{thm:main}, (ii), it is enough to prove that
$$\liminf_{n\to\infty}\frac{\|u\|_{L^2(A_n,\dm)}}{\|\varphi_n u\|_{2,m}}
	=0\,.
$$
Let $B_n:=\mathrm{supp}(\varphi_n)$. Notice that if $\tilde{J}(n):=\|u\mathbf{1}_{B_{n-1}}\|_{2,m}$, then
$$\tilde{J}\leq J,
$$
as follows from the assumption that $\varphi_n|_{B_{n-1}}\equiv 1$. Hence, $\tilde{J}$ is also subexponentially bounded. Notice also that since $A_n\subset\mathrm{supp}(\varphi_{n+1})\setminus \{\varphi_{n-1}=1\}\subset  B_{n+2}\setminus B_{n-2}$, the estimate
$$
\|u\|_{L^2(A_n,\dm)}
	\leq \tilde{J}(n+3)-\tilde{J}(n-1)\,.
$$
holds. Also,we have
$$\|\varphi_nu\|_{2,m}
	\geq \tilde{J}(n-1)
$$
implying
$$
\frac{\|u\|_{L^2(A_n,\dm)}}{\|\varphi_n u\|_{2,m}}
	\leq \frac{\tilde{J}(n+3)-\tilde{J}(n-1)}{\tilde{J}(n-1)}
	=\frac{\tilde{J}(n+3)}{\tilde{J}(n-1)}-1\,.
$$
Since $\tilde{J}$ is subexponentially bounded, according to \cite[Lemma 4.2]{BoLeSt09}, one has
$$\liminf_{n\to\infty}\frac{\tilde{J}(n+3)}{\tilde{J}(n-1)}
	=1\,,
$$
hence
$$\liminf_{n\to\infty}\frac{\|u\|_{L^2(A_n,\dmu)}}{\|\varphi_n u\|_{2,\mu}}
	=0\,.
$$
We conclude that $\lambda\in\sigma(H)$ from (ii) in Theorem~\ref{thm:main}.
\end{proof}
%%%%%%%%%%%%%%%%%%

\medskip

As a special case of Corollary~\ref{Cor:Schnol3}, one can recover in our setting the following classical result of A.~Boutet de Monvel, D.~Lenz and P.~Stollmann \cite[Theorem~4.4]{BoLeSt09}, which generalizes earlier results due to {\`E}. Schnol \cite{Shn57} and B. Simon \cite{Sim81}: we recall according to the terminology introduced therein that a function $u : \Omega\to\R$ is called {\em subexponentially bounded} if for a fixed point $p\in\Omega$ and for any $\alpha>0$, $\mathrm{e}^{-\alpha d_A(p,\cdot)}u\in L^2(\Omega,d\mu)$. Here, $d_A(x,y)$ is the intrinsic metric with respect to $A$, defined by
$$
d_A(x,y)
	:=\sup\{|\psi(x)-\psi(y)|\,;\,\psi\in \mathcal{C}(\Omega)\cap W^{1,2}_{loc}(\Omega),\,\langle A\nabla \varphi,\nabla\varphi\rangle\leq 1\,a.e.\}.
$$
We make the assumption that $d_A(x,y)$ is a metric (that is, the above supremum is always finite) and induces the same topology on $\Omega$. As a consequence of these assumptions (c.f. \cite[p. 193]{BoLeSt09}), one has
$$
|\nabla d_A(\cdot, y)|_A\leq 1\,a.e.,
$$
for every $y\in \Omega$.

%%%%%%%%%%%%%%%%%%
\begin{corollary}

Let $\Omega$ be complete, $H$ be of the form \eqref{Eq:Schro}, and let $u\in W^{1,2}_{\loc}(\Omega)$ be a generalized eigenfunction of the operator $H$, associated with the eigenvalue $\lambda\in \R$. Assume that the intrinsic metric $d_A$ determines the same topology on $\Omega$ as the usual one. If $u$ subexponentially bounded, then $\lambda\in\sigma(H)$.

\end{corollary}
%%%%%%%%%%%%%%%%%%

%%%%%%%%%%%%%%%%%%
\begin{proof}
Without loss of generality, one can assume that $\Omega$ has infinite diameter for the distance $d_A$ (otherwise, $u$ subexponentially bounded implies that $u$ is in $L^2$, hence $\lambda$ is an eigenvalue of $H$). Let $\varphi(x)=d_A(x,p)$, where $d_A$ is the intrinsic metric and $p\in\Omega$ is a fixed point. The assumption that the topology on $\Omega$ endowed with the intrinsic distance is the same as the original one, implies that $\varphi$ is an exhaustion. By the above,

\begin{equation}\label{grad}
|\nabla \varphi|_A\leq 1\,a.e.
\end{equation}
almost everywhere. We choose $r_n=n$, $R_n=n+b$, for some fixed value $b\in(0,1)$, and let $\varphi_n:=\psi_{r_n,R_n}$, where we recall that $\psi_{r_n,R_n}$ is defined by

$$\psi^n_{r,R}(x)=\begin{cases}
\frac{R-\varphi_n(x)}{R-r},\qquad &x\in \overline{A(r,R)},\\
1,\qquad &x\in \Omega_r,\\
0,\qquad &x\notin \Omega_R.
\end{cases}$$
Since for every $n\in\mathbb{N}$,

$$|\nabla \varphi_n|_A
	\leq C_b\,\,\,a.e.\,,
$$
the weak Hardy inequality \eqref{Eq:WH} is trivially satisfied. According to \cite[Lemma 4.3]{BoLeSt09}, the function 

$$J(n)
	:=\|\varphi_n u\|_{2,m}\leq \int_{B(p,n+1)}|u|^2\dm
$$
is subexponentially bounded. Also, since for $k=n-1,n,n+1$, $\nabla \varphi_k$ has support included in $A_n=\mathrm{supp}(\varphi_{n+1}(1-\varphi_{n-1}))$, one has

$$\frac{\sum_{k=n-1}^{n+1}\left(\int_\Omega |u|^2 |\nabla \varphi_k|_A^2\right)^{1/2}}{\|\varphi_n u\|_{2,m}}
	\leq C_b \frac{\|u\|_{L^2(A_n,\dm)}}{\|\varphi_n u\|_{2,m}}\,.
$$
It was shown in the proof of Corollary~\ref{Cor:Schnol3} that

$$\liminf_{n\to\infty}\frac{\|u\|_{L^2(A_n,\dm)}}{\|\varphi_n u||_{2,m}}
	=0\,.
$$
Hence,

$$\liminf_{n\to\infty} \frac{\sum_{k=n-1}^{n+1}\left(\int_\Omega |u|^2|\nabla \varphi_k|^2_A\dm\right)^{1/2}}{\|\varphi_n u\|_{2,m}}
	\leq C_b\; \liminf_{n\to\infty}\frac{\|u\|_{L^2(A_n,m)}}{\|\varphi_n u\|_{2,m}}=0\,,
$$
and the result follows from Corollary~\ref{Cor:Schnol3}.
\end{proof}
%%%%%%%%%%%%%%%%%%

%%%%%%%%%%%%%%%%%%%%%%%%%%%%%%%%%%%%%%%%%%%%%%%%%%%%%%%%%%%%%%%%%%%%%%%%%%%%%%%%%%%%%
%\subsection{The subexponential growth condition and comparison with existing results}
%\label{Ssec:SubExpComp}
%%%%%%%%%%%%%%%%%%%%%%%%%%%%%%%%%%%%%%%%%%%%%%%%%%%%%%%%%%%%%%%%%%%%%%%%%%%%%%%%%%%%%

%%%%%%%%%%%%%%%%%%

\section{Discussion - Examples}
\label{Sec:Examples}

We first provide an example showing that our results are stronger than those of \cite{BoLeSt09}. The main result in \cite{BoLeSt09} (Theorem 4.4 therein) is applicable provided the generalized eigenfunction $u$ is such that for any $\alpha>0$, $e^{-\alpha d(x_0,\cdot)}u\in L^2$ ($u$ is ``subexponentially bounded'' in the terminology of \cite{BoLeSt09}). 

%%%%%%%%%%%%%%%%%%
\begin{example}\label{Ex:parabolic}

There is a complete Riemannian manifold for which (i) in Theorem \ref{thm:main} gives $0\in \sigma(\Delta)$, but the associated generalized eigenfunction $u\equiv 1$ is not subexponentially bounded.

\end{example}
%%%%%%%%%%%%%%%%%%

%%%%%%%%%%%%%%%%%%
\begin{proof}
Note that we take the sign convention for $\Delta$ that makes it a non-negative operator. A first class of examples follows from a result of R. Brooks (\cite{B81}): it is enough to consider a covering of a compact Riemannian manifold by an amenable group having exponential growth (such groups are known to exists, for instance the so-called lamplighter group). 

\medskip

A different class of examples can be obtained as follows: it is enough to find a complete parabolic Riemannian manifold having exponential volume growth. Indeed, the parabolicity means that $\Delta$ is critical, with ground state $1$. By (i) in Theorem \ref{thm:main}, $0\in \sigma(\Delta)$. On the other hand, clearly $1$ does not have subexponential growth. Two different examples of such manifolds can be found in \cite{V89} and \cite{R92} (the second one even having constant curvature $=-1$).
\end{proof}

\medskip

Next, one has the following example showing that (i) in Theorem \ref{thm:main} does not always follow from (ii):

%%%%%%%%%%%%%%%%%%
\begin{example}\label{Ex:i-ii}

There is a complete Riemannian manifold and an admissible cut-off sequence $\{\varphi_n\}_{n\in\mathbb{N}}$ for $(\Delta,1)$, such that $0\in \sigma(\Delta)$, and for $u\equiv 1$, (i) with $\varphi\equiv 1$ in Theorem~\ref{thm:main} holds but not (ii). More precisely,

$$\liminf_{n\to\infty}\frac{\|u\|_{L^2(A_n,\dm)}}{\|\varphi_n u\|_{2,m}}>0.$$

\end{example}
%%%%%%%%%%%%%%%%%%

%%%%%%%%%%%%%%%%%%
\begin{proof}
We take $M=\R^2$, endowed with the Euclidean metric $g_{eucl}$, $\dm$ the Lebesgue measure, and $H=\Delta$. Clearly, $0$ belongs to the spectrum of $\Delta$. Moreover, it is well-known that $\R^2$ is parabolic: therefore, we take the admissible cut-off sequence $\{\varphi_n\}_{n\in\mathbb{N}}$ to be the usual null-sequence of $(\R^2,g_{eucl})$, constructed with the method of Section \ref{Ssec:Null} with the Evans potential given by $\log(|x|)$, $|x|\geq1$. In this construction, one takes $\varphi_n=\psi_{r_n,R_n}$ with $r_n=e^{2n}$ and $R_n=e^{2n+1}$, and one easily checks that the conditions (H1)--(H4) on $r_n$ and $R_n$ are satisfied. Clearly, $\|\varphi_n u\|_{2,m}\to \infty$, and $\mathbf{a}(\varphi_n,\varphi_n)=\int_{M}|\nabla \varphi_n|^2\dm\to 0$ since $\{\varphi_n\}_{n\in\mathbb{N}}$ is a null-sequence, so that (i) with $\varphi\equiv1$ of Theorem~\ref{thm:main} is satisfied. We now show that (ii) for the sequence $\{\varphi_n\}_{n\in\mathbb{N}}$ and $u\equiv 1$ is not satisfied. It is easily seen that $\mu(A_n)=\mu(\bar{A}_n)$ with $\bar{A}_n=\{z\in \R^2\,;\,e^{e^{2n}}\leq |z|\leq e^{e^{2n+1}}\}$. Hence,

$$\left(\frac{\|u\|_{L^2(A_n,\dm)}}{\|\varphi_n u\|_{2,m}}\right)^2
	\geq \frac{m(B(0,e^{e^{2n+1}}))-m(B(0,e^{e^{2n}}))}{m(B(0,e^{e^{2n+1}}))}
	=1-\frac{m(B(0,e^{e^{2n}}))}{m(B(0,e^{e^{2n+1}}))}
	\geq 1- \left(\frac{e^{e^{2n}}}{e^{e^{2n+1}}}\right)^2\,,
$$
which is easily seen to tend to $1$, as $n\to \infty$.
\end{proof}
%%%%%%%%%%%%%%%%%%

\medskip

%\begin{example}
%\label{Ex:Euclidean}
%Spectrum of $\Delta+V$ on $\R^n$, for $V$ small perturbation of $\Delta$; is it true that there are generalized eigenfunctions $u$ with

%$$u\sim e^{i\xi\cdot x}?$$
%\red{ask Yehuda about Agmon's notes}
%\end{example}

In the next example, we explain how our results allow one to recover the spectrum of the Laplacian on the hyperbolic space $\mathbb{H}^n$, despite it having exponential volume growth.

\begin{example}
\label{Ex:Hyperbolic}
Let $\mathbb{H}^n$ be the real hyperbolic space of dimension $n$ given by the Poincar\'e disk model, and $\Delta$ be the Laplace-Beltrami operator acting on functions on $\mathbb{H}^n$. We will denote by $r=d(0,\cdot)$ where the distance is the hyperbolic distance. Consider the operator

$$H=\Delta-\left(\frac{n-1}{2}\right)^2.$$
We are going to see that one can use (i) in Theorem \ref{thm:main}, in order to show that $\sigma(H)=[0,\infty)$, hence $\sigma(\Delta)=[\left(\frac{n-1}{2}\right)^2,\infty)$ (of course, this latter result is well-known, but our point is to show that Theorem \ref{thm:main} can be used also in situations where the volume growth is exponential). Let $\lambda>0$. According to \cite{GO05}, there exists a radial generalized eigenfunction $u_\lambda(r)$ for $H$ associated to the eigenvalue $\lambda$, moreover

$$|u_\lambda(r)|\lesssim e^{-\frac{n-1}{2}r}.$$
Let us now study the Green functions of the operator $H$. By using the symmetries of the hyperbolic space, it can be shown that the Green functions must be radial. In \cite[Proposition 7.1]{GO05}, it is also proved that radial, $H$-harmonic functions on $\mathbb{H}^n\setminus\{0\}$ are linear combination of two functions called $h_s(r)$ (which is globally defined on $\mathbb{H}^n$ and positive), and $g_s(r)$ (which is $H$-superharmonic on $\mathbb{H}^n$), $s=\frac{n-1}{2}$. This already implies that the operator $H$ is subcritical. Moreover, by \cite[Proposition 7.2 and (7.13)]{GO05},

$$g_s(r)\asymp e^{-\frac{n-1}{2}r},\quad r\to \infty,$$
and 

$$\lim_{r\to \infty}h_s(r)e^{\frac{n-1}{2}r}=+\infty.$$
It follows that $g_s(r)$ has minimal growth as an $H$-harmonic function, hence one obtains that the Green functions for $H$ are two-sided bounded by $e^{-\frac{n-1}{2}r}$ as $r\to \infty$. Let us now take $W\in C_0^\infty$, a potential such that $H+W$ is critical with ground state $\varphi$. Since $W$ is compactly supported, $\varphi$ must have minimal growth for $H$, hence

$$\varphi\asymp e^{-\frac{n-1}{2}r}.$$
It follows that

$$|u_\lambda|\lesssim \varphi.$$
We can thus apply Corollary \ref{Cor:BP} and get that $\lambda\in \sigma(H)$.

\end{example}

\begin{example} 

We present an example for which (i) in Theorem \ref{thm:main} is applicable but not \cite{BP17}. Let $m:\R\to [1,\infty)$ be smooth and such that $m(x)=|x|^3$ for $|x|\geq1$. Let $\Gw=\R$, endowed with the measure $\dm=m(x)dx$, and let $H$ be the operator associated with the quadratic form

$$Q(u)=\int_{\R}|u'(x)|^2\,dm(x).$$
Explicitly,
%%%%%%%%%%%%%%%%%%
$$
H=-\frac{1}{m(x)}\frac{d}{dx}\left(m(x)\frac{d}{dx}\cdot \right)=-\frac{d^2}{dx^2}-\frac{m'(x)}{m(x)}\frac{d}{dx}.
$$
For $|x|\geq 1$, $H=-\frac{d^2}{dx^2}-\frac{3}{x}\frac{d}{dx}$, which is the radial part of the Laplacian on $\R^4$ endowed with the Euclidean metric. The function $|x|^{-2}$ is $H$-harmonic and has minimal growth as $|x|\to \infty$, since the Green function of the Laplacian on $\R^4$ is (up to a multiplicative constant) equal to $|x|^{-2}$. Since $Q\geq Q_{\Delta}$, the subcriticality of $\Delta$ on $\R^4$ implies that $H$ is subcritical, and its Green functions $G(x,y)$ satisfy $G(x,y)\asymp |x|^{-2}$ as $|x|\to \infty$. Let $W\in C_0^\infty((-1,1))$ be such that $H+W$ is critical; then, the corresponding ground state $h$ satisfies $h(x)\asymp |x|^{-2}$ as $|x|\to \infty$. Let $\lambda\geq0$, and consider the following regular linear ODE:

$$\frac{d^2u}{dx^2}+\frac{m'(x)}{m(x)}\frac{du}{dx}+(\lambda+W(x)) u(x)=0.$$
This ODE admits a $2$-dimensional vector space of global solutions; let us pick $u_\lambda$ such a non-trivial globally defined solution. 
Since the support of $W$ is included in $(-1,1)$, $u_\lambda$ is a solution of the following Sturm-Liouville equation for $|x|\geq 1$:
%%%%%%%%%%%%%%%%%%
$$
u_\lambda''+\frac{3}{x} u_\lambda'+\lambda u_\lambda =0,\quad |x|\geq1.
$$
%%%%%%%%%%%%%%%%%%
 General solutions to this equation are given (see \cite[Page~117, Eq.~6.80]{Bowman58}) by
%%%%%%%%%%%%%%%%%%
\begin{align*}
y(x) = \frac{1}{x} \big( A J_{-1}(\sqrt{\lambda} x) + B Y_{-1}(\sqrt{\lambda} x) \big), \quad A,B\in \R,
\end{align*}
%%%%%%%%%%%%%%%%%%
where $J_n$ is the Bessel function and $Y_n$ is the Bessel function of the second kind (Neumann functions). According to \cite[Page~364, 9.2.1, 9.2.2]{AbramowitzIrene64}, the long time behavior of the Bessel functions for $|x|\to\infty$ is
%%%%%%%%%%%%%%%%%%
\begin{align*}
J_{-1}(\sqrt{\lambda} x) & = \sqrt{\frac{2}{\pi \sqrt{\lambda} |x|}}\, \left( 
		\cos\left(\sqrt{\lambda} |x|+\frac{\pi}{4}\right)
		+
		O\big( |x|^{-1} \big)
	\right),\\
Y_{-1}(\sqrt{\lambda} x) & = \sqrt{\frac{2}{\pi \sqrt{\lambda} |x|}}\, \left( 
		\sin\left(\sqrt{\lambda} |x|+\frac{\pi}{4}\right)
		+
		O\big( |x|^{-1} \big)
	\right).
\end{align*}
%%%%%%%%%%%%%%%%%%
Thus, as $x\to \pm \infty$,
%%%%%%%%%%%%%%%%%%
\begin{align*}
u_\lambda(x) = |x|^{-3/2}\left(A_\pm\cos\left(\sqrt{\lambda} |x|+\frac{\pi}{4}\right)+B_\pm \sin\left(\sqrt{\lambda} |x|+\frac{\pi}{4}\right) +O(|x|^{-1})\right),
\end{align*}
for some $A_\pm,B_\pm\in\R$. Since the ground state $h$ satisfies 

$$h(x)\asymp |x|^{-2},\quad |x|\to \infty,$$
there does not exists any constant $C>0$ such that

$$|u_\lambda| \leq C h.$$
Therefore, \cite{BP17} is not applicable. Next, we show that the criterion (i) in Theorem \ref{thm:main} is applicable, and give that $\lambda\in \sigma(H)$. First, we define

$$L=h^{-1}(H+W)h,$$
and we notice that $\varphi(x)=h^{-1}(x)=x^2$ is an Evans potential for $L$ (see Section \ref{Ssec:Null}). We construct $\{\varphi_n\}_{n\in\N}$, an admissible cut-off sequence for $(H+W,h)$; by Lemma \ref{Lem:AdSeq-crit}, this is equivalent to building $\{\frac{\varphi_n}{h}\}_{n\in\N}$ admissible cut-off sequence for $(L,1)$. We achieve this by following the method in Section \ref{Ssec:Null}; more precisely, we let

$$\frac{\varphi_n}{h}=\psi_{r_n,R_n}$$
and we choose the sequences $\{r_n\}_{n\in N}$ and $\{R_n\}_{n\in N}$ so that (H1)--(H4) hold. We will choose $1\leq r_n<R_n$; given the definition of $\psi_{r,R}$ in Section \ref{Ssec:Null}, we get $|\frac{d}{dx} \psi_{r,R}|=\frac{2|x|}{R-r}$, hence for $\dmu=h^2\dm$,
\begin{align*}
\int_{A(r_n,R_n)} \left\langle \frac{d}{dx} \psi_{r_n,R_n}, \frac{d}{dx} \psi_{r_n,R_n}\right\rangle \dmu
	= &\frac{1}{(R_n-r_n)^2} \int_{r_n<x^2<R_n} 4|x|^2\frac{1}{|x|^4} |x|^3 dx\\
	= &\frac{4}{(R_n-r_n)^2} \int_{r_n<x^2<R_n} |x| dx\\
	= &\frac{4}{R_n-r_n}
\end{align*}
Define $r_n:= e^{2n}$ and $R_n:=e^{2n+1}$. A straightforward computation shows that (H1)--(H4) are then satisfied, therefore $\{\frac{\varphi_n}{h}\}_{n\in\N}$ is an admissible cut-off sequence. As a consequence of the construction (see Section 5), we also get that $\{\frac{\varphi_n}{h}\}_{n\in \N}$ is a null sequence for $L$, which implies that

\begin{equation}\label{Eq:Null_ex}
\lim_{n\to \infty}\int_\Omega \left|\frac{d}{dx} \left(\frac{\varphi_n}{h}\right)\right|^2\,h^2\dm=0.
\end{equation}
By definition of $A_n:=\supp(\frac{\varphi_{n+1}}{h}(1-\frac{\varphi_{n-1}}{h}))$, we get 

$$A_n\subseteq \{r_{n-1}<x^2<R_{n+1}\}
	\subseteq \{e^{2n-2} <x^2<e^{2n+3}\}
	=  \{e^{n-1} <|x|<e^{n+3/2}\}
$$

Thus, for $x\in A_n$ and $n$ large, we get 

$$\left|\frac{u_\lambda(x)}{h(x)}\right| \lesssim \left|\frac{x^2}{x^{3/2}}\right| = \sqrt{x} \lesssim e^{\frac{n}{2}}
$$
Recall that for $|x|\geq1$,

$$u_\lambda=x^{-3/2}\left(A_\pm\cos\left(\sqrt{\lambda} x+\frac{\pi}{4}\right)+B_\pm \sin\left(\sqrt{\lambda} x+\frac{\pi}{4}\right)\right) +E(x)=v(x)+E(x),$$
with $E(x)=O(|x|^{-5/2})$; obviously, $||E||_{2,m}<\infty.$ By definition of an admissible cut-off sequence, we have

$$\|\frac{\varphi_n}{h}u_\lambda\|_{2,m} ^2\geq \int_{1\leq x^2\leq e^{2n}} |v(x)|^2 |x|^3 dx+C =\int_{1\leq x\leq e^{n}} |v(x)|^2 x^3 dx +\int_{-e^{n}\leq x\leq -1} |v(x)|^2 |x|^3 dx+C.
$$
We estimate from below the first integral (for positive $x$), the computations for the second one being similar. Writing

$$A_+\cos\left(\sqrt{\lambda} x+\frac{\pi}{4}\right)+B_+ \sin\left(\sqrt{\lambda} x+\frac{\pi}{4}\right)=\kappa \cos(\sqrt{\lambda}x+\theta),\quad x\geq1,$$
for some constants $\kappa$ and $\theta$, we obtain

$$v(x)=\kappa x^{-3/2} \cos(\sqrt{\lambda}x+\theta),\quad x\geq1.$$
Therefore,

$$\int_{1\leq x\leq e^{n}} |v(x)|^2 x^3 dx\geq \kappa^2\int_{1\leq x\leq e^n}\cos^2(\sqrt{\lambda}x+\theta)\,dx=\frac{\kappa^2}{\sqrt{\lambda}}\int_{\sqrt{\lambda}+\theta \leq y\leq \sqrt{\lambda}e^n+\theta }\cos^2(y)\,dy$$
Using that

$$\int_{c\leq y\leq r}\cos^2(y)\,dy \sim \frac{r}{2},\quad r\to \infty,$$
we get, as $n\to \infty$,

$$e^n \lesssim \int_{1\leq x\leq e^{n}} |v(x)|^2 x^3 dx.$$
Consequently, for $n$ large

$$e^{n/2}\lesssim \|\frac{\varphi_n}{h}u_\lambda\|_{2,m}.$$
Therefore,

$$\frac{\max_{A_n}\frac{|u_\lambda|}{h}}{\|\frac{\varphi_n}{h}u_\lambda\|_{2,m}}\lesssim 1.$$
Hence, taking \eqref{Eq:Null_ex} into account, we obtain

$$\lim_{n\to \infty}\frac{\max_{A_n}\frac{|u_\lambda|}{h}}{\|\frac{\varphi_n}{h}u_\lambda\|_{2,m}}\int_\Omega \left|\frac{d}{dx} \left(\frac{\varphi_n}{h}\right)\right|^2\,h^2\dm=0.$$
Thus, (i) with $\varphi=h$ in Theorem \ref{thm:main} is satisfied, and $\lambda\in \sigma(H)$.

\end{example}

%%%%%%%%%%%%%%%%%%
\begin{center}{\bf Acknowledgments} \end{center}
%%%%%%%%%%%%%%%%%%%%%%%%%%%%%%%%%%%%%%%%%%%%%%%
The authors are grateful to thank Yehuda Pinchover for useful remarks and fruitful discussions. S. Beckus acknowledges the support of the Israel Science Foundation (grants No. 970/15) founded by the Israel Academy of Sciences and Humanities.

%%%%%%%%%%%%%%%%%%%%%%%%%%%%%%%%%%%%%%%%%%%%%%%%%%%%%
%%%%%%%%%%           References            %%%%%%%%%%
%%%%%%%%%%%%%%%%%%%%%%%%%%%%%%%%%%%%%%%%%%%%%%%%%%%%%
%\bibliographystyle{amsalphasort}
%\bibliography{references}

\begin{thebibliography}{9999999999}
\bibitem[AI64]{AbramowitzIrene64}
M.~Abramowitz, I.~A.~Stegun, \emph{Handbook of mathematical functions with formulas, graphs, and mathematical tables}, National Bureau of Standards Applied Mathematics Series, Vol. 55, 1964.
              
\bibitem[Agm83]{Agm83}
S.~Agmon, \emph{On positivity and decay of solutions of second order elliptic
  equations on {R}iemannian manifolds}, 19--52, Methods of functional analysis and
  theory of elliptic equations ({N}aples, 1982), Liguori, Naples, 1983.
  
\bibitem[Anc02]{Anc}
A.~Ancona, \emph{Some results and examples about the behavior of harmonic functions and Green's functions with respect to second order elliptic operators}, Nagoya Math. J. {\bf 165} (2002), 123--158
  
\bibitem[BP17]{BP17}
S.~Beckus and Y.~Pinchover, \emph{Shnol-type theorem for the Agmon ground state}, to appear in J. Spectr. Theory, arXiv:1706.04869, 2017.

\bibitem[BM95]{BiMo95}
M.~Biroli and U.~Mosco, \emph{A {S}aint-{V}enant type principle for {D}irichlet
  forms on discontinuous media}, Ann. Mat. Pura Appl. (4) {\bf 169} (1995), 125--181.

\bibitem[BH91]{BouleauHirsch91}
N.~Bouleau and F.~Hirsch, \emph{Dirichlet forms and analysis on {W}iener
  space}, De Gruyter Studies in Mathematics, vol.~14, Walter de Gruyter \& Co.,
  Berlin, 1991.

\bibitem[BdMLS09]{BoLeSt09}
A.~Boutet~de Monvel, D.~Lenz, and P.~Stollmann, \emph{Sch'nol's theorem for
  strongly local forms}, Israel J. Math. {\bf 173} (2009), 189--211.

\bibitem[BdMS03]{BoSt03}
A.~Boutet~de Monvel, P.~Stollmann, \emph{Eigenfunction expansions for
  generators of {D}irichlet forms},  J. Reine Angew. Math. {\bf 561} (2003), 131--144.

\bibitem[Bo58]{Bowman58}
F.~Bowman, \emph{Introduction to {B}essel functions}, Dover Publications Inc., New York, 1958.

\bibitem[Br54]{Bro54}
F.~E.~Browder, \emph{The eigenfunction expansion theorem for the general self-adjoint singular elliptic partial differential operator. {I}. {T}he analytical foundation}, Proc. Nat. Acad. Sci. U. S. A. {\bf 40} (1954), 454--459. 

\bibitem[B81]{B81}
R.~Brooks, \emph{The fundamental group and the spectrum of the Laplacian}, Comm. Math. Helv. {\bf 56} (1981), 581--598. 

\bibitem[CL14]{CL14}
N.~Charalambous and Z.~Lu.: {On the spectrum of the Laplacian}, Math. Ann. {\bf359} (2014),  211--238.

\bibitem[CK91]{ChKa91}
I.~Chavel, L.~Karp, \emph{Large time behavior of the heat kernel: the parabolic {$\lambda$}-potential alternative},  Comment. Math. Helv. {\bf 66} (1991), 541--556.

\bibitem[CFKS87]{CyconFroeseKirschSimon87}
H.~L. Cycon, R.~G. Froese, W.~Kirsch, and B.~Simon, \emph{Schr\"odinger
  operators with application to quantum mechanics and global geometry}, study
  ed., Texts and Monographs in Physics, Springer-Verlag, Berlin, 1987.

\bibitem[Dav95]{Davies95}
E.~B. Davies, \emph{Spectral theory and differential operators}, Cambridge
  Studies in Advanced Mathematics, vol.~42, Cambridge University Press,
  Cambridge, 1995.

%\bibitem[Dev12]{Dev12}
%B.~Devyver, \emph{On the finiteness of the {M}orse index for {S}chr\"odinger
%  operators}, 249--271, Manuscripta Math., vol.~139, no.~1-2, 2012.

\bibitem[DDV98]{DeDuVi98}
Y.~Dermenjian, M.~Durand, and I.~Viorel.: {Spectral analysis of an acoustic multistratified perturbed cylinder}.
  Commun. Partial Differential Equations {\bf 23} (1998), 141--169.

\bibitem[DFP14]{DeFrPi14}
B.~Devyver, M.~Fraas, and Y.~Pinchover, \emph{Optimal {H}ardy weight for
  second-order elliptic operator: an answer to a problem of {A}gmon}, J.~Funct.
  Anal., {\bf 266} (2014), 4422--4489.


  \bibitem[EMO92]{EMO92} A.~Erd\'elyi, W.~Magnus, F.~Oberhettinger, and F.G.~Tricomi, \emph{Higher transcendental functions. {V}ols. {I}, {II}}, McGraw-Hill, New York, 1953.

  \bibitem[FR92]{R92} J.L.~Fernandez, J.M.~Rodriguez, \emph{Area growth and Green's function of Riemann surfaces}, {\em Ark. Mat.} {\bf 30} (1992), no. 1, 83--92.


\bibitem[FLW14]{FrLeWi14}
R.~L.~Frank, D.~Lenz, and D.~Wingert,
  \emph{Intrinsic metrics for non-local symmetric Dirichlet forms and applications to spectral theory},
   {\em J.~Funct. Anal.} {\bf 266} (2014), 4765--4808.

\bibitem[Fuk80]{Fukushima80}
M.~Fukushima, \emph{Dirichlet forms and {M}arkov processes}, North-Holland
  Mathematical Library, vol.~23, North-Holland Publishing Co., Amsterdam-New
  York; Kodansha, Ltd., Tokyo, 1980.

\bibitem[FOT94]{FukushimaOshimaTakeda94}
M.~Fukushima, Y.~\=Oshima, and M.~Takeda, \emph{Dirichlet forms and symmetric
  {M}arkov processes}, De Gruyter Studies in Mathematics, vol.~19, Walter de
  Gruyter \& Co., Berlin, 1994.
  
  \bibitem[GT01]{GT}
  D.~Gilbarg and N.~Trudinger, \emph{Elliptic partial differential equations of second order}, Reprint of the 1998 edition, Classics in Mathematics. Springer-Verlag, Berlin, 2001.

\bibitem[GO05]{GO05}
S.~Grellier, and J.-P.~Otal,
  \emph{Bounded eigenfunctions in the real hyperbolic space},
  Int. Math. Res. Not. {\bf 62} (2005), 3867--3897.
   
\bibitem[G99]{Grig}
A.~Grigor'yan, Analytic and geometric background of recurrence and non-explosion of the Brownian motion on Riemannian manifolds, {\em Bull. Amer. Math. Soc.} (N.S.) {\bf 36} (1999), no. 2, 135--249.

\bibitem[HK11]{HaKe11} S.~Haeseler, and M.~Keller, \emph{Generalized solutions
  and spectrum for Dirichlet forms on graphs. Random walks, boundaries and spectra},
  Progr. Probab., \textbf{64}, Birkh\"{a}user/Springer Basel AG, Basel (2011), 181--199.

\bibitem[HKM93]{HeinonenKilpelainenMartio93}
J.~Heinonen, T.~Kilpel\"ainen, and O.~Martio, \emph{Nonlinear potential theory
  of degenerate elliptic equations}, Oxford Mathematical Monographs, The
  Clarendon Press, Oxford University Press, New York,  Oxford Science
  Publications, 1993.

\bibitem[HKMW13]{HuaMasKelWoj13}
X.~Huang, M.~Keller, J.~Masmune, and R.~K.~Wojciechowski , \emph{A note on
  self-adjoint extensions of the Laplacian on weighted graphs}, J. Funct. Anal. {\bf 265} (2013), 1556--1578.

%\bibitem[HS96]{HiSi69}
%P.~D.~Hislop, and I.~M.~Sidal, \emph{Introduction to spectral theory. With applications to Schr\"odinger operators},
%   Applied Mathematical Sciences, Springer-Verlag, New York, 1996.

\bibitem[KL12]{KeLe12}
M.~Keller, and D.~H.~Lenz, \emph{Dirichlet forms and stochastic completeness of graphs and subgraphs}, J.~Reine Angew. Math. {\bf 666} (2012), 189--223.

\bibitem[KPP16]{KePiPo16}
M.~Keller, Y.~Pinchover, and F.~Pogorzelski, \emph{Optimal hardy inequalities
  for Schr\"odinger operators on graphs}, Comm. Math. Phys. {\bf 358}, (2018), no. 2, 767--790.

\bibitem[KL14]{KrLu14}
D.~Krej\v{c}i\v{r}\'{i}k, and Z.~Lu.: {Location of the essential spectrum in curved quantum layer}. J.~Math. Phys. {\bf 55 (8)}, 083520 (2014).

\bibitem[K05]{Ku05}
   P.~Kuchment, \emph{Quantum graphs. {II}. {S}ome spectral properties of quantum and combinatorial graphs}, J. Phys. A {\bf 38}, (2005), no. 22, 4887--4900.

\bibitem[LSV09]{LeStVe09}
D.~H.~Lenz, P.~Stollmann, and I.~Veseli{\'c}, \emph{The Allegretto-Piepenbrink theorem for strongly local forms}, Documenta Mathematica  {\bf 14} (2009), 167--189.

\bibitem[LSV11]{LeStVe11}
D.~H.~Lenz, P.~Stollmann, and I.~Veseli{\'c}, \emph{Generalized eigenfunctions
  and spectral theory for strongly local Dirichlet forms}, Oper. Theory Adv. Appl. {\bf 214} (2011),  83--106.
  
\bibitem[LSW63]{LSW}
W.~Littman, G.~Stampacchia, H.~F.~Weinberger, \emph{Regular points for elliptic equations with discontinuous coefficients}, Ann. Scuola Norm. Sup. Pisa (3) {\bf 17} 1963 43--77.

\bibitem[MR92]{MaRoeckner92}
Z.~M. Ma and M.~R\"ockner, \emph{Introduction to the theory of (nonsymmetric) {D}irichlet forms}, Universitext, Springer-Verlag, Berlin, 1992.

\bibitem[P88]{Pi88} Y.~Pinchover, \emph{On positive solutions of second-order elliptic equations, stability results and classification}, Duke Math.
J. {\bf 57} (1988), 955--980.

\bibitem[P92]{P92}
Y.~Pinchover,
 \emph{Large time behavior of the heat kernel and the behavior of the
  {G}reen function near criticality for nonsymmetric elliptic operators},
J.~Functional Analysis {\bf 104} 54--70, 1992.

\bibitem[P04]{P04}
Y.~Pinchover, {\em Large time behavior of the heat kernel}, J.~Functional Analysis {\bf 206} (2004), 191--209.

\bibitem[P07]{Pi07} Y.~Pinchover, {\em Topics in the theory of positive
solutions of second-order elliptic and parabolic partial
differential equations}, in ``Spectral Theory and Mathematical
Physics: A Festschrift in Honor of Barry Simon's 60th Birthday",
eds. F.~Gesztesy, et al., 329--356, Proceedings of Symposia in Pure
Mathematics {\bf 76}, American Mathematical Society, Providence,
RI, 2007.

\bibitem[P07a]{Pi07a}  Y.~Pinchover, {\em A Liouville-type theorem for Schr\"odinger
operators}, Comm. Math. Phys. {\bf 272} (2007), 75--84.

\bibitem[PT06]{PiTi06}
Y.~Pinchover and K.~Tintarev, \emph{A ground state alternative for singular
  {S}chr\"odinger operators}, J.~Funct. Anal., {\bf 230} (2006), 65--77.

\bibitem[Pin95]{Pinsky95}
R.~G. Pinsky, \emph{Positive harmonic functions and diffusion}, Cambridge
  Studies in Advanced Mathematics, vol.~45, Cambridge University Press,
  Cambridge, 1995.

\bibitem[Shn57]{Shn57}
\`{E}.~\`{E}. \v{S}hnol', \emph{On the behavior of the eigenfunctions of {S}chr\"odinger's
  equation}, Mat. Sb. (N.S.), {\bf 42, (84)} (1957), 273--286; erratum, {\bf 46 (88)} (1957), 259.

\bibitem[Shu92]{Shu92} M.~.A.~Shubin, \emph{Spectral theory of elliptic operators on noncompact manifolds},
  35--108, M\'{e}thodes semi-classiques, Vol. 1 (Nantes, 1991), Ast\'{e}risque, vol.~207, 1992.

\bibitem[Sim81]{Sim81}
B.~Simon, \emph{Spectrum and continuum eigenfunctions of {S}chr\"odinger
  operators}, J.~Funct. Anal., {\bf 42} (1981), 347--355.
  
  \bibitem[Sim82]{Sim82}
  B.~Simon, \emph{Schr\"odinger semigroups}, Bull. AMS, {\bf 7} (3) (1982), 447--526.
  
\bibitem[Sim93]{Sim93}
B.~Simon, \emph{Large time behavior of the heat kernel: on a theorem of {C}havel and {K}arp}, Proc. Amer. Math. Soc., {\bf 118} (1993), 513--514.

\bibitem[Sto01]{Stollmann01}
P.~Stollmann, \emph{Caught by disorder}, Bound states in random media, Progress in Mathematical Physics,
  vol.~20, Birkh\"auser, Boston, Inc., Boston, MA, 2001.
  
\bibitem[Var89]{V89} N.Th.~Varopoulos, \emph{Small time Gaussian estimates of heat diffusion kernels. I. The semigroup technique}, Bull. Sci. Math. {\bf 113}, (1989), no. 3, 253--277.  
  
  
\end{thebibliography}

\providecommand{\bysame}{\leavevmode\hbox to3em{\hrulefill}\thinspace}
\providecommand{\MR}{\relax\ifhmode\unskip\space\fi MR }
\providecommand{\MRhref}[2]{%
  \href{http://www.ams.org/mathscinet-getitem?mr=#1}{#2}
}
\providecommand{\href}[2]{#2}

\end{document}